\documentclass[11pt,twoside,reqno,centertags,
tikz,border=3mm]{amsart}

\usepackage{amsmath,amsthm,amsfonts,amssymb,amsrefs,
ascmac 
}
\usepackage{enumerate,color}

\usepackage{etoolbox} 

\makeatletter

\@addtoreset{equation}{section}
\makeatother

\theoremstyle{plain} 
\newtheorem{thm}{Theorem}[section]
\newtheorem{lem}[thm]{Lemma}
\newtheorem{cor}[thm]{Corollary}

\theoremstyle{definition} 
\newtheorem{definition}[thm]{Definition}
\newtheorem{rem}[thm]{Remark}

\newcommand{\Z}{\mathbb {Z}}

\newcommand{\R}{\mathbb{R}}

\newcommand{\ep}{\varepsilon}

\newcommand{\supp}{\rm{supp}}

\def\div{ \hbox{\rm div}\,  }

\def\supp{\, \hbox{\rm supp}\,  }
\def\id{\hbox{\rm Id}}

\def\cC{{\mathcal C}}

\let\tilde=\widetilde

\newcommand{\eqsp}[1]{{\begin{equation}\begin{aligned}#1\end{aligned}\end{
equation}}}

\title[Well-posedness and time-decay estimates of CNSK]{Global well-posedness and time-decay estimates  of the compressible Navier-Stokes-Korteweg system in critical Besov spaces}
\date\today

\author[N. Chikami]{Noboru Chikami}
\address[N. Chikami]
{Graduate School of Engineering Science, Osaka University,
1-3 Machikaneya-cho, Toyonaka, Osaka, 560-8531, Japan.}
\email{noboru.chikami@sigmath.es.osaka-u.ac.jp}

\author[T. Kobayashi]{Takayuki Kobayashi}
\address[T. Kobayashi]
{Graduate School of Engineering Science, Osaka University,
1-3 Machikaneyamacho, Toyonaka, Osaka, 560-8531, Japan.}
\email{kobayashi@sigmath.es.osaka-u.ac.jp}

\begin{document}

\begin{abstract}
We consider the compressible Navier-Stokes-Korteweg system 
describing the dynamics of a liquid-vapor mixture with diffuse interphase. 
The global solutions are established under linear stability conditions 
in critical Besov spaces. 
In particular, the sound speed may be greater than or equal to zero. 
By fully exploiting the parabolic property of the linearized system for all 
frequencies, we see that there is no loss of derivative usually induced by the 
pressure for the standard isentropic compressible Navier-Stokes system. 
This enables us to apply Banach's fixed point theorem to show the 
existence of global solution. 
Furthermore, we obtain the optimal decay rates of the global solutions 
in the $L^2(\R^d)$-framework.  
\end{abstract}
\maketitle


\section{Compressible Navier-Stokes-Korteweg system}
We consider the following barotropic compressible Navier-Stokes-Korteweg system endowed with an internal local capillarity:
\begin{equation}
\label{CNSKu}
\left\{\begin{aligned}
&\partial_t \rho + \div (\rho u) = 0, \\
&\partial_t (\rho u) + \div (\rho u\otimes u) + \nabla (P(\rho) )
    =  \div (\tau(\rho, \nabla u) + K(\rho)), \\
&(\rho,u)|_{t=0} = (\rho_0,u_0), 
\end{aligned}\right.
\end{equation} 
where $(t,x) \in \R_+ \times \R^d,$ $d\ge2$, 
$\rho = \rho(t,x)\in\R_+$ and $u=u(t,x)\in\R^d$ are unknowns 
representing the fluid density and the velocity vector field, respectively. 
The viscous strain tensor $\tau$, in general, is given by
\begin{equation*}
\tau = \tau(\rho, D u)= 2\mu(\rho) D(u) + \lambda(\rho) \div u \ \id , 
\end{equation*}
with $\id$ denoting an identity matrix and 
and $D(u)$ designating the deformation tensor defined by 
\begin{equation*}
D(u):=\frac{1}{2}(\nabla u+ \,^t (\nabla u)) \quad \text{with} \quad 
(\nabla u)_{ij} := \partial_i u^j,
\end{equation*}
where $\,^t (\nabla u)$ denotes the transposed matrix of $\nabla u.$
For simplicity we assume that 
the viscosity coefficients $\lambda$ and $\mu$ are given constants 
satisfying $\mu>0$ and $\lambda + 2 \mu>0$. 
We write $\mathcal{L}u = \div(2\mu Du) + \nabla (\lambda \div u)$.  
When $\mu$ and $\lambda$ are constants, then we see that 
$\div (\tau) = \mathcal{L}u = \mu \Delta u +(\mu+\lambda) \nabla \div u$.  
The pressure is assumed to be barotropic, i.e., 
$P$ is a sufficiently smooth function of $\rho$. 

In the general Korteweg system, 
the capillary term is given by $\div (K(\rho)),$ where 
$K(\rho)$ is the Korteweg tensor defined by 
\begin{equation}\nonumber
K(\rho) = \frac{\kappa(\rho)}{2} (\Delta \rho^2 - |\nabla \rho|^2) \id - \kappa(\rho) \nabla \rho \otimes \nabla\rho, 
\end{equation}
where $\nabla \rho \otimes \nabla\rho$ stands for the 
tensor product $(\partial_j \rho \, \partial_k \rho )_{jk}$.
The coefficient $\kappa$ may depend on $\rho$ in general. 
Here, we assume the case of positive constant coefficient 
$\kappa(\rho) \equiv \kappa$. 
We assume that there is no vacuum, and $\rho$ tends to a positive constant $\rho_*$ at spatial infinity. 

The system describes the dynamics of a liquid-vapor 
mixture in the setting of the Diffuse Interface approach: 
between the two phases lies a thin region of continuous transition 
and the phase changes are described through the variations of the density 
(with for example a {\sl Van der Waals pressure}). 
Historically, the system is first proposed by J. D. Van der Waals (1893), 
and later Korteweg tensor is introduced by D. J. Korteweg (1901).   
The system considered here is a version reformulated by J. E. Dunn and J. Serrin \cite{DunSer1985} in 1985. 

An important features of System \eqref{CNSKu} is that 
the pressure is non-monotone in general. 
Simply put, $P'(\rho_*)$ is no longer a positive number, 
which inevitably leads to linear instability for some value of $\rho_*.$ 
As such, most mathematical results (in the whole space) 
until now deal with the condition that $P'(\rho_*)$ is strictly positive, 
except for local-in-time results. 

There are numerous works dedicated to the study of the 
system with the capillarity term. 
For known well-posedness results for System \eqref{CNSKu}, 
global smooth solutions are established by \cite{HatLi1994, HatLi1996}. 
Decay estimates along with the analysis of diffusive wave phenomenon 
are considered by T. Kobayashi and K. Tsuda \cite{KobTsu2018}. 
The latter also obtains the existence of time-periodic solutions along with 
their stability in \cite{Tsu2016}. 
See also \cite{Kot2014, TanZha2014} for the existence and decay results. 
In the critical framework, R. Danchin and B. Desjardins 
\cite{DanDes2001} first obtain the well-posedness in critical Besov spaces. 
Especially, they consider a general non-monotone pressure to 
show local existence of the solution. 
For results on non-local capillary terms and convergence to various models, 
we refer to the works by F. Charve and B. Haspot 
\cite{ChaHas2011, ChaHas2013, Cha2014}. 
In \cite{TanGao2016}, T. Tang and H. Gao obtain 
the global solution for System \eqref{CNSKu} 
in $\mathbb{T}^d$ under a non-monotone pressure law. 

Two difficulties usually arise in the analysis of 
the compressible viscous fluids. One is the loss of derivative 
induced by the pressure term and the other is the appearance of the 
second order nonlinear term. Both of them preclude the application 
of standard semi-group theory and necessitate the laborious energy estimates. 
In contrast to the compressible viscous fluids without the capillary term, 
however, System \eqref{CNSKu} is known to have one mathematical upside, 
which is that there is no longer a loss of derivative due to the parabolic smoothing effect for both the velocity and the density. 

We may check that a {\sl scaling invariance} holds for System \eqref{CNSKu}:
namely a family of rescaled functions $(\rho,u)\to (\rho_{\nu},u_{\nu})$ with
\begin{equation}\label{CNSK:sc.inv}
\rho_{\nu} (t,x)=\rho(\nu^2 t,\nu x) \quad\text{and}\quad 
u_{\nu} (t,x)=\nu u(\nu^2 t,\nu x) 
\end{equation}
leaves \eqref{CNSKu} invariant provided that $P$ is changed into $\nu^2 P$. 
The pressure term, to some extent, can be regarded 
as a lower order term in our setting and therefore 
it is possible to solve the system in the {\sl scale-critical function spaces}, 
i.e., the spaces in which the norm is kept invariant under the transforms 
\eqref{CNSK:sc.inv}. 
Moreover, we clarify in this paper that it suffices to assume that the pressure 
is merely non-negative, i.e., $P'(\rho*)\ge0,$ in order to obtain global solutions 
in critical Besov spaces. 
We establish a priori estimates for the linearized system that is 
{\sl uniform} with respect to the parameter $P'(\rho_*)\ge0$ 
(See Lemma \ref{lCNSKp}). 
We also show the decay estimates of such global small solutions 
under additional $L^1(\R^d)$-type regularity assumption. 

We remark that although we focus our analysis in the 
$L^2(\R^d)$-setting, it is not difficult to extend our decay results 
to the $L^p(\R^d)$-framework in the spirit of \cite{DanXu2017}. 
In fact, F. Charve, R. Danchin and J. Xu 
have recently shown the existence of solution for System \eqref{CNSK} in 
$L^p(\R^d)$-framework along with the Gevrey analyticity 
of the solution \cite{ChDaXu2018}. As a corollary of the analyticity, they obtain 
decay estimates of higher order norms with respect to a critical data. 
Note that contrary to \cite{ChDaXu2018} 
our decay estimates deal with the {\sl decay rates of critical 
norms with respect to a $L^1(\R^d)$ data}, which naturally 
requires more labour and cannot be deduced from their results. 
Furthermore, our results for the critical case $P'(\rho_*)=0$ 
seems completely new. 

\subsection{Notations and function spaces}
Hereafter, we denote by $L^p(\R^d),$ $1\le p\le\infty,$ 
the standard Lebesgue spaces on 
the $d$-dimensional Euclidean space $\R^d$, 
and by $\ell^p(\Z)$ the set of sequences with summable $p$-th powers. 
Unless otherwise stated, we assume $d\ge2$ throughout the paper. 
We abbreviate $\|\cdot\|_{L^p(\R^d)} = \|\cdot\|_{L^p}$ 
whenever the dimension $d$ is not relevant, 
and denote the time-space Lebesgue norm by 
$\|\cdot\|_{L^q(0,T;L^p)} = \|\cdot\|_{L^q_TL^p}.$ 
An infinitely differentiable, complex-valued function $f$ on $\R^d$ 
is called a Schwartz function if for every pair of multi-indices $\alpha$ 
and $\beta$ there exists a positive constant $C_{\alpha,\beta}$ such that 
$\sup_{x\in\R^d} |x^\alpha\partial^\beta f| = C_{\alpha,\beta} < \infty.  $
We denote the set of all Schwartz functions on $\R^d$ by $\mathcal{S}(\R^d)$
and its topological dual by $\mathcal{S}'(\R^d).$ Elements in $\mathcal{S}'(\R^d)$
are called tempered distributions. 
We define the fractional derivative operator by 
\begin{equation}\nonumber
\Lambda^s := (-\Delta)^{s/2} = \mathcal{F}^{-1} |\cdot|^s \mathcal{F}
\end{equation}
for $s\in\R,$ where $\mathcal{F}u = \widehat u$ denotes the Fourier transform of $u \in \mathcal{S}'(\R^d).$

We define a subspace of $\mathcal{S}'(\R^d)$ which lays a basic foundation of our analysis. For the details, we refer to \cite{BCD2011}. 
\begin{definition}[\cite{BCD2011} P22]
\label{d:S_h}
We denote by $\mathcal{S}'_h(\R^d)$ the space of tempered distributions $u$ such that 
$ \displaystyle
\lim_{\lambda\to\infty} \|\mathcal{F}^{-1} [\theta(\lambda \cdot) \widehat u]\|_{L^\infty} = 0 
$
for any $\theta$ in $\mathcal{S}(\R^d).$
\end{definition}

In what follows, we denote by $C$ a generic constant that may change from line to line. 
We write $A\cong B$ whenever there exists a positive constant $c$ such that $c^{-1} A \le B \le c A.$ 
We also denote $A\lesssim B$ whenever there exists arbitrary harmless constant 
$C$ such that $A \le C B$.
Constants depending on parameters $a,b,\ldots,c$ are denote by $C(a,b,\ldots,c)$. 
We agree that $p'$ always denotes the H\"older conjugate of $p,$ i.e., $\frac1{p}+\frac1{p'}=1,$ with the convention $\frac{1}{\infty}=0.$ 
Vectors in this article are understood as column vectors in general; 
however, we omit the sign of transpose $^t$ whenever it is clear from 
the context or simply irrelevant. 
We express a function $u$ of variables $(t,x)$ as $u$, 
$u(t)$ or $u(t,x)$ depending on circumstances. 

\subsubsection{Besov spaces}
Let $\{ \phi_j \}_{j\in \Z}$ be a homogeneous Littlewood-Paley dyadic decomposition. 
Namely, let $\widehat\phi \in \mathcal{S}(\R^d)$ be a non-negative, radially symmetric function that satisfies
\[
\supp \widehat{\phi} \subset \{ \xi \in \R^d; 2^{-1} \le |\xi |\le 2\}
\quad
\text{and} \quad \sum_{j\in \Z} \widehat{\phi}(2^{-j}\xi) = 1 
\quad\text{for all} \quad \xi\neq 0.
\]
Setting $\widehat{\phi_j} (\xi) :=\widehat{\phi} (2^{-j}\xi)$
and $\displaystyle \widehat{\Phi_{j-1}}(\xi):= 1 - \sum_{k \ge j} \widehat{\phi_k} (\xi)$
for $j\in \Z$, we have
\begin{equation}\nonumber
\supp \widehat{\Phi_{j-1}} 
      \subset\{\xi\in\R^d \,;\, |\xi| \le 2^j \} 
  \quad\text{and}\quad
  \supp \widehat{\phi_j} 
      \subset\{\xi\in\R^d \,;\, 2^{j-1} \le |\xi| \le 2^{j+1} \}
\end{equation}
for all $j\in\Z$. 
Given an element $u$ in $\mathcal{S}'_h(\R^d)$ and a (homogeneous) Littlewood-Paley decomposition $\{\phi_j\}_{j\in\Z},$ 
we define the (homogeneous) Littlewood-Paley localization operator by 
\begin{equation}\label{def:h-LP-proj}
\dot \Delta_j u:= \mathcal{F}^{-1}[\widehat\phi_j \widehat u] 
\end{equation}
and the frequency cut-off operator by 
$\displaystyle \dot S_j u :=  \sum_{j'\le j} \Delta_{j'} u = \mathcal{F}^{-1} [\widehat\Phi_j \widehat u]$ for $j\in\Z.$ 

\begin{definition}[Homogeneous Besov space]
For $s\in\R$ and $1\le p, \sigma \le \infty$ we define the homogeneous Besov spaces $\dot{B}^{s} _{p,\sigma}(\R^d)$ as follows:
\begin{equation}\nonumber
\dot{B}^s_{p,\sigma}(\R^d) := \{u\in \mathcal{S}'_h(\R^d)  \ ; \ \|u \|_{\dot{B}^{s}_{p,\sigma}(\R^d)}<\infty\}, \quad
\|u \|_{\dot{B}^{s}_{p,\sigma}(\R^d)} 
:= \left\| \left\{2^{js} \| \dot\Delta_j u\|_{L^p}\right\}_{j\in\Z} \right\|_{\ell^\sigma}.
\end{equation}
For notational convenience, we abbreviate 
$\|\cdot\|_{\dot{B}^{s}_{p,\sigma}(\R^d)} = \|\cdot\|_{\dot{B}^{s}_{p,\sigma}}.$
\end{definition}
More properties of Besov spaces may be found in e.g. \cites{BCD2011}.
\medbreak
For some fixed $j_0\in\Z,$ we
denote by $\displaystyle u_h := \dot S_{j_0} u$ the low frequencies of $u,$  
 and by $u_h:=u-u_h$ the high frequencies  of $u.$ 
We also need the notation
\begin{equation}\label{eq:LH1}
\|u\|^l_{\dot B^s_{2,1}}:= \sum_{j\leq j_0} 2^{js}\|\dot\Delta_ju\|_{L^2}
\quad\hbox{and}\quad
\|u\|^h_{\dot B^\sigma_{2,1}}:= \sum_{j\geq j_0-1} 2^{j\sigma}\|\dot\Delta_j u\|_{L^2}.
\end{equation}
The small overlap between low and high frequencies are to ensure that 
\begin{equation}\nonumber
\|u_l\|_{\dot B^s_{2,1}}\leq C \|u\|^l_{\dot B^s_{2,1}}\quad\hbox{and}\quad
\|u_h\|_{\dot B^\sigma_{2,1}}\leq C \|u\|^h_{\dot B^\sigma_{2,1}}.
\end{equation}

Lastly, we introduce the so-called Chemin-Lerner spaces. 
\begin{definition}[Chemin-Lerner spaces, \cites{ChLe95, BCD2011}]
\label{d:CheLer.sp}
For $t_1,t_2>0$, $s\in\R$, $1\le \sigma,\rho\le \infty$, we set 
$$
\|u\|_{\widetilde{L^\rho}(t_1,t_2 ; \dot B^s_{p,\sigma})}
:= \left\| \{2^{js} 
\|\dot\Delta_j u\|_{L^\rho(t_1,t_2;L^p)} \}_{j\in\Z} \right\|_{\ell^{\sigma}}.
$$
We then define the space $\widetilde{L^\rho}(t_1,t_2;\dot B^s_{p,\sigma}(\R^d))$ 
as the set of tempered distributions $u$ over 
$(t_1,t_2)\times \R^d$ such that 
$\lim_{j\to -\infty} \dot S_j u = 0$ in $\mathcal{S}'((t_1,t_2)\times\R^d)$
and $\|u\|_{\widetilde{L^\rho}(t_1,t_2;\dot B^s_{p,\sigma})} < \infty$. 
When $t_1=0$ and $t_2 = T,$ we denote 
$\widetilde{L^\rho_T}(\dot B^s_{p,\sigma}(\R^d))=\widetilde{L^\rho}(0,T; \dot B^s_{p,\sigma}(\R^d))$ with the shortened expression of the norm $\|\cdot\|_{\widetilde{L^\rho_T}(\dot B^s_{p,\sigma})}.$
\end{definition}
In what follows, we agree that
\begin{equation}\label{d:tldC.sp}
\widetilde\cC([0,T];\dot B^s_{2,1}):=\bigl\{v\in\cC([0,T];\dot B^s_{2,1})\:; \|v\|_{\widetilde{L^\infty_T}(\dot B^{s}_{2,1})}<\infty\bigr\}\cdotp
\end{equation}
\subsection{Momentum formulation and main results}
Setting $m :=\rho u,$ we recast System \eqref{CNSKu} in 
the momentum formulation, which reads as follows:
\begin{equation}\label{CNSK}\tag{CNSK}
\left\{\begin{aligned}
&\partial_t \rho + \div m = 0,\\
&\partial_t m + \div(\rho^{-1} m\otimes m) + \nabla P
    =  \mathcal{L}(\rho^{-1} m) +  \div(K(\rho)), \\
&(\rho,m)|_{t=0} = (\rho_0,m_0).
\end{aligned}\right.
\end{equation}
The clear difference from the velocity formulation is that now the 
mass conservation law is linear while the capillary term is nonlinear. 
Also notable is that there is $\rho^{-1}$ inside the second-order term, 
thereby requiring more regularity to the density {\sl a priori}. 
This poses no obstacle in the actual analysis since, as is well-known, 
there is a parabolic smoothing effect for the density in System \eqref{CNSK}. 
In this paper, we perform our analysis in the momentum formulation. 
In fact, the momentum formulation turns out to be 
essential for our analysis for the case $\gamma=0$; 
Particularly, the fact the the nonlinear terms may be written in divergence form 
greatly helps our analysis in $d=3,$ as we see in the proof of Lemma \ref{l:CNSK.nonl.est.tld}. 

\subsubsection{Global existence}
Throughout this paper, we set $\gamma:= P'(\rho_*).$ 
We first state a result of well-posedness for $d\ge2$ and $\gamma\ge0,$
which is contained in \cite{DanDes2001}. 
Under smallness of the initial data, the solution 
is globally extended when $\gamma>0$. 
\begin{thm}[\cite{DanDes2001}]\label{t:L2.WP}
Let $d\ge 2$ and let $P$ be a smooth function such that $\gamma\ge0$. 
For sufficiently small initial data such that 
$$
((\gamma+\Lambda)(\rho_0-\rho_*), m_0) \in \dot{B}^{\frac{d}{2}-1}_{2,1}(\R^d),
$$ 
there exists a time $T>0$ and a unique 
solution $(\rho,m)$ to \eqref{CNSK} satisfying 
\begin{equation}\nonumber
((\gamma+\Lambda) (\rho-\rho_*), m)\in \widetilde{\cC}([0,T];\dot B^{\frac{d}{2}-1}_{2,1}(\R^d)) 
		\cap L^1(0,T;\dot B^{\frac{d}{2}+1}_{2,1}(\R^d)).
\end{equation}
Moreover, if $\gamma>0$ then $T=\infty$ and there exists a constant $C$ 
depending on $d,$ $\mu,$ $\lambda,$ $\rho_*$ and $P$ such that  
\begin{equation}\nonumber\begin{aligned}
\|((\gamma+\Lambda) (\rho-\rho_*),  m)
	\|_{\widetilde{L^\infty}(0,\infty;\dot{B}^{\frac{d}{2}-1}_{2,1})}
&+ \|\Lambda^2 ((\gamma+\Lambda) (\rho-\rho_*),  m)
	\|_{L^1(0,\infty;\dot{B}^{\frac{d}{2}-1}_{2,1})} \\
&\quad\le C \|((\gamma+\Lambda)(\rho_0-\rho_*), m_0)\|_{\dot{B}^{\frac{d}{2}-1}_{2,1}}. 
\end{aligned}\end{equation}
\end{thm}
\begin{rem}
As observed in \cite{DanDes2001}, as far as a local solution is concerned, 
the smallness of the initial data and the condition $\gamma\ge0$ 
may be removed. 
\end{rem}

The question remains whether we may obtain a global solution 
when the sound speed is zero, i.e., when $\gamma = 0.$ 
It turns out that restricting the dimension to $d\ge3$, we have the global 
well-posedness for the case under an additional 
low-frequency assumption. 
\begin{thm}\label{t:L2.WP.szero}
Let $d\ge 3$ and let $P$ be a smooth function such that $\gamma=0$. 
For sufficiently small initial data such that 
$$
(\Lambda  (\rho_0-\rho_*), m_0) 
\in \dot{B}^{\frac{d}{2}-3}_{2,1}(\R^d)\cap \dot{B}^{\frac{d}{2}-1}_{2,1}(\R^d),
$$ 
there exists a unique solution $(\rho,m)$ to \eqref{CNSK} satisfying 
\begin{equation}\nonumber
(\Lambda (\rho-\rho_*), m)
\in \widetilde{\cC}([0,\infty);\dot{B}^{\frac{d}{2}-3}_{2,1}(\R^d)\cap
			\dot B^{\frac{d}{2}-1}_{2,1}(\R^d)) 
		\cap \widetilde{L^1}(0,\infty;
		\dot{B}^{\frac{d}{2}-1}_{2,1}(\R^d)\cap 
		          \dot B^{\frac{d}{2}+1}_{2,1}(\R^d)).
\end{equation}
Moreover, there exists a positive constant $C$ 
depending on $d,$ $\mu,$ $\lambda,$ $\rho_*$ and $P$ such that  
\begin{equation}\nonumber\begin{aligned}
\|(\Lambda (\rho-\rho_*),  m)
	\|_{\widetilde{L^\infty}(0,\infty;\dot{B}^{\frac{d}{2}-1}_{2,1})}
&+ \|\Lambda^2 ((\Lambda (\rho-\rho_*),  m)
	\|_{\widetilde{L^1}(0,\infty;\dot{B}^{\frac{d}{2}-1}_{2,1})} \\
&\quad\le C \|((\Lambda (\rho_0-\rho_*), m_0)\|_{\dot{B}^{\frac{d}{2}-1}_{2,1}}. 
\end{aligned}\end{equation}
\end{thm}
\begin{rem}
In both Theorem \ref{t:L2.WP} and Theorem \ref{t:L2.WP.szero}, 
the space $\widetilde{\cC}(0,T;\dot{B}^{s}_{2,1}(\R^d)$ is defined by 
Definition \ref{d:tldC.sp}. 
\end{rem}
Let us remark that although we focus our analysis for the barotropic case, 
our technique can be also applied to non-isentropic cases. 
The results can also be naturally extended 
in the settings of torus and variable capillarity and viscosity coefficients. 

\subsubsection{Decay estimates}
Next, we clarify the decay rates of the solution constructed by 
Theorem \ref{t:L2.WP} and Theorem \ref{t:L2.WP.szero}. 

We first state a result for the case $\gamma>0.$  
Setting $\langle t\rangle := (1+t),$ we introduce an auxiliary norm 
\begin{multline}\label{d:CNSK.Dt}
D(t) := \sup_{s\in(-\frac{d}{2},\frac{d}{2}+1]} 
\| \langle\tau\rangle^{\frac{1}{2}(s+\frac{d}{2})}
((\gamma+\Lambda)(\rho-\rho_*), m)\|_{L^\infty_t(\dot B^{s}_{2,1})}^l \\
 + \| \tau^{\alpha}
\Lambda^2 ((\gamma+\Lambda)(\rho-\rho_*), m)\|_{\widetilde{L^\infty_t}(\dot B^{\frac{d}{2}-1}_{2,1})}^h. 
\end{multline}
Here and in what follows, we agree that the notation $\langle\tau\rangle^\sigma f,$ $\sigma\in\R,$ 
designates the function $(\tau,x)\mapsto \langle\tau\rangle^\sigma f(\tau,x).$ 
The notations $\|\cdot \|_{L^\infty_t(\dot B^{s}_{2,1})}^l$ and 
$\|\cdot \|_{\widetilde{L^\infty_t}(\dot B^{\frac{d}{2}-1}_{2,1})}^h$ 
are defined analogously to \eqref{eq:LH1}. 

\smallbreak
Under additional assumptions on the low frequencies of the  initial data,
one may obtain time-decay estimates that are reminiscent of 
those of the standard compressible Navier-Stokes equations 
(see \cite{DanXu2017}). 
\begin{thm}\label{t:CNSK.dcy}
Let $\gamma>0.$ 
Let the data $(a_0,m_0)$ satisfy the assumptions of Theorem \ref{t:L2.WP}. 
There exist a frequency-threshold $j_0$ only depending on $d,$ 
$\mu,$ $\lambda,$ $\kappa$ and $\gamma$ and 
a positive constant $c$ so that if in addition 
\begin{equation}\nonumber
D_0:=   \|\rho_0-\rho_*\|_{\dot B^{-\frac{d}{2}}_{2,\infty}}^l
	+ \|m_0\|_{\dot B^{-\frac{d}{2}}_{2,\infty}}^l \le c
\end{equation}
then the global solution $(a,m)$ given by Theorem \ref{t:L2.WP} 
satisfies for all $t\ge0,$
\begin{equation}\label{t:CNSK.dcy:smllsol}
D(t) \le 
C\bigl(  D_0+ \|(\Lambda (\rho_0-\rho_*),m_0)\|_{\dot B^{\frac{d}{2}-1}_{2,1}}^h \bigr)
\end{equation}
where $D(t)$ is defined by \eqref{d:CNSK.Dt} and 
$\alpha := \frac{d}{2}+\frac{1}{2}-\varepsilon$ for sufficiently small $\ep>0.$ 
\end{thm}

We next state the result for the case $\gamma=0.$ We set 
\begin{equation}\label{d:CNSK.Dt.s0}
\tilde D(t) := \sup_{s\in(-\frac{d}{2},\frac{d}{2}+1]} 
\| \langle\tau\rangle^{\frac{1}{2}(s+\frac{d}{2})}
(\Lambda (\rho-\rho_*), m)\|_{L^\infty_t(\dot B^{s}_{2,1})}^l 
 + \| \tau^{\alpha}
\Lambda^2 (\Lambda (\rho-\rho_*), m)\|_{\widetilde{L^\infty_t}( \dot B^{\frac{d}{2}-1}_{2,1})}^h. 
\end{equation}
For the case $\gamma=0,$ we have the following. 
\begin{thm}\label{t:CNSK.dcy.s0}
Let $\gamma=0.$ Let the data $(\rho_0, m_0)$ 
satisfy 
\begin{equation}\label{t:CNSK.dcy.s0.asmp.id}
(\Lambda (\rho_0-\rho_*), m_0) \in \dot{B}^{-\frac{d}{2}}_{2,\infty}(\R^d) 
\cap \dot{B}^{\frac{d}{2}-1}_{2,1}(\R^d).
\end{equation}
There exist a frequency-threshold $j_0$ only depending on $d,$ 
$\mu,$ $\lambda$ and $\kappa$ and 
a positive constant $c$ so that if in addition 
\begin{equation}\nonumber
\tilde D_0:=   \|(\rho_0-\rho_*)\|_{\dot B^{1-\frac{d}{2}}_{2,\infty}}^l
	+ \|m_0\|_{\dot B^{-\frac{d}{2}}_{2,\infty}}^l \le c
\end{equation}
then the global solution $((\rho-\rho_*),m)$ given by Theorem \ref{t:L2.WP.szero} 
satisfies for all $t\ge0,$
\begin{equation}\label{t:CNSK.dcy.s0:smllsol}
\tilde D(t) \le 
C\bigl( \tilde D_0+ \|(\Lambda (\rho_0-\rho_*),m_0)\|_{\dot B^{\frac{d}{2}-1}_{2,1}}^h \bigr)
\end{equation}
where $\tilde D(t)$ is defined by \eqref{d:CNSK.Dt.s0} and 
$\alpha := \frac{d}{2}+\frac{1}{2}-\varepsilon$ for sufficiently small $\ep>0.$ 
\end{thm}
\begin{rem}
Note that since $d\ge3,$ we have $- \frac{d}{2} \le \frac{d}{2}-3,$ 
which implies $\|u\|_{\dot{B}^{\frac{d}{2}-3}_{2,1}}^l
\lesssim\|u\|_{\dot{B}^{-\frac{d}{2}}_{2,\infty}}^l$ for $d\ge3.$
Thus, the regularity assumption \eqref{t:CNSK.dcy.s0.asmp.id} 
in Theorem \ref{t:CNSK.dcy.s0} is stronger than 
that of Theorem \ref{t:L2.WP.szero}, 
and the corresponding small global solution is bound to naturally arise. 
\end{rem}

The rest of the paper is organized as follows: in the second section, we 
study the linearized problem of \eqref{CNSK}. 
The third section is devoted to the proof of Theorem \ref{t:L2.WP}
and Theorem \ref{t:L2.WP.szero}. 
In the fourth section, we prove Theorem \ref{t:CNSK.dcy} and 
Theorem \ref{t:CNSK.dcy.s0}. 
In the Appendix, we collect some properties of Besov spaces.

\section{Linearized problem}
Linearization around a constant state $(\rho_*,0)$ gives 
\begin{equation}\label{CNSKp:rho.m}
\left\{\begin{aligned}
&\partial_t a + \div m = 0 ,\\
&\partial_t m - \frac{1}{\rho_*} \mathcal{L} m  
    -\kappa \rho_* \nabla\Delta a
   + \gamma \nabla a
    =  \mathcal{L}\left( Q(a) m\right) 
    + \div(K(a)). \\
&\qquad
     - \frac{1}{\rho_*}  \div(m\otimes m) 
    -\div\left( Q(a)m\otimes m \right) 
    +\left( P'(a+\rho_*)- \gamma \right) \nabla a,
\end{aligned}\right.
\end{equation}
where we put $\gamma := P'(\rho_*),$ 
$a := \rho-\rho_*$ and $Q(a) :=\frac{a}{\rho_*(a+\rho_*)}$.
We split $m$ into compressible and incompressible components, 
i.e., $m=w+\mathcal{P}^{\bot}u,$ with $w:=\mathcal{P}m$
where $\mathcal{P}$ and $\mathcal{P}^{\bot}$ 
are the projectors onto divergence-free and 
potential vector-fields, respectively ($w:=({\rm Id}+\nabla\div(-\Delta)^{-1}) m$). 
Setting $v:=\Lambda^{-1} \div m=\Lambda^{-1} \div\mathcal{P}^{\bot}m,$ 
we have 
\begin{equation}\label{lCNSK.qvw}
\left\{\begin{aligned}
&\partial_t a + \Lambda v = 0 , \\
&\partial_t v - \underline\nu \Delta v - \gamma \Lambda a - \kappa\rho_* \Lambda^3 a = 0, \\
&\partial_t w - \underline\mu \Delta w = 0, \\
\end{aligned}\right.
\end{equation}
where 
\begin{equation}\label{d:CNSKnubar}
\underline\nu := \frac{\nu}{\rho_*},\quad 
\nu := 2\mu+\lambda,\quad 
\underline\mu := \frac{\mu}{\rho_*}.
\end{equation}
We investigate the properties of the above linearized system. 

\subsection{Maximal regularity estimates in Besov spaces}
Let us consider the following standard heat equation 
with a constant diffusive coefficient $\nu>0$:
\begin{equation}\label{nuH}
\left\{\begin{array}{cl}
\partial_t u -\nu\Delta u = f,\quad  &t>0, \quad x\in\R^d, \\[1ex]
u|_{t=0} = u_0, &x\in\R^d,
\end{array}\right.
\end{equation}
where the outer force $f :\R_+\times\R^d \to \R$ and the initial data
$u_0 : \R^d \to \R$ are given functions or tempered distributions. 
Recall the following estimate for a linear heat equation, 
which can be found in literatures such as 
\cite{BCD2011} P157 and \cite{Da2015} P7. 
\begin{lem}\label{l:heatbesov}
Let $T>0$, $s\in\R$ and $1\le p, r, \sigma \le\infty.$
Assume that $u_0 \in \dot B^s_{p,\sigma}(\R^d)$ and 
$f\in \widetilde{L^r_T}(\dot B^{s-2+\frac{2}{r}}_{p,\sigma}(\R^d))$. 
Then \eqref{nuH} has a unique solution $u$ in 
$\widetilde{L^\infty_T}(\dot B^{s}_{p,\sigma}(\R^d))
\cap \widetilde{L^r_T}(\dot B^{s+\frac{2}{r}}_{p,\sigma}(\R^d))$
and there exists a constant $C$ depending only on $q$ and $r$ such that for all $q \in [r,\infty],$ we have  
\begin{equation}
\nonumber
\nu^{\frac{1}{q}} 
\|u \|_{\widetilde{L^q_T}(\dot B^{s+\frac{2}{q}}_{p,\sigma})} 
\le C\Big( \|u_0\|_{\dot B^{s}_{p,\sigma}} 
+ \nu^{\frac{1}{r}-1} \|f\|_{\widetilde{L^r_T}(\dot B^{s-2+\frac{2}{r}}_{p,\sigma})} \Big).
\end{equation}
If in addition $\sigma$ is finite, then $u$ belongs to $C([0,T);\dot B^{s}_{p,\sigma}(\R^d))$. 
\end{lem}
\noindent Our aim here is to prove a similar estimate for the linearized system of \eqref{CNSK}. 

\smallbreak

Two difficulties usually arise in the analysis of 
the compressible viscous fluids. One is the loss of derivative 
induced by the pressure term and the other is the appearance of the 
second order nonlinear term. Both of them preclude the application 
of standard semi-group theory and necessitate laborious energy estimates. 
In contrast to the compressible viscous fluids without the capillary term, 
System \eqref{CNSK} is known to have a pleasant mathematical upside, 
which is that there is no longer a loss of derivative due to the parabolic smoothing effect for the density. 

From hereon, we consider a linear hyperbolic-parabolic equation 
\begin{equation}\label{lCNSKp}
\left\{\begin{aligned}
&\partial_t a + \div m = f, \\
&\partial_t m - \frac{1}{\rho_*} \mathcal{L}m  +  \gamma\nabla a 
	- \kappa \rho_* \nabla\Delta a =  g \\
\end{aligned}\right.
\end{equation}
with initial data $(a_0,m_0)$ 
and derive an a priori estimate in Besov spaces 
using the method of Lyapunov in the frequency space. 
\begin{lem}\label{l:lFH}
Let $T>0$, $s\in\R$, $1\le r, \sigma\le \infty$, $\gamma\ge0$,
$$
((\gamma+\Lambda) a_0, m_0) \in \dot{B}^s_{2,\sigma}(\R^d)
\quad\text{and}\quad
((\gamma+\Lambda) f, g) 
	\in \widetilde{L^r}(0,T;\dot{B}^{s-2+\frac{2}{r}}_{2,\sigma}(\R^d)).
$$
Suppose that $(a,m)$ is a solution for \eqref{lCNSKp} such that 
\begin{equation}\nonumber
((\gamma+\Lambda) a, m) 
\in \widetilde{L^q}(0,T; \dot{B}^{s+\frac{2}{q}}_{2,\sigma}(\R^d)) 
\quad\text{for}\quad q \in [r,\infty]
\end{equation} 
with initial data $(a_0,m_0).$
Then there exists some positive constant $C$ depending only on $\rho_*,$ 
$\lambda$, $\mu$, $\kappa$, $d$, $q$, $r$ and $\sigma$ 
such that for all $q \in [r,\infty],$ 
the following estimate hold: 
\begin{equation}\nonumber
\|((\gamma+\Lambda) a, m) \|_{\widetilde{L^q_T}(\dot{B}^{s+\frac{2}{q}}_{2,\sigma})} 
\le C\left( \|((\gamma+\Lambda) a_0, m_0) \|_{\dot{B}^{s}_{2,\sigma}} 
+\|((\gamma+\Lambda) f, g) \|_{\widetilde{L^r_T}(\dot{B}^{s-2+\frac{2}{r}}_{2,\sigma})} \right).
\end{equation}
for all $0 <t \le T.$
\end{lem}
\begin{rem}
Note that the above constant does not depend the parameter $\gamma$ 
and holds for both cases $\gamma>0$ and $\gamma=0.$ 
It is clear from the above estimates that one needs to impose a 
stronger control on the low frequencies at a linear level 
when $\gamma>0$, which is in accordance with the known result for the 
isentropic compressible fluid. 
\end{rem}
\begin{proof}[Proof of Lemma \ref{l:lFH}]
The argument below is inspired by \cite{Da2015}. 
However, the resulting smoothing effects, as expected, are strikingly 
better than those of the compressible Navier-Stokes system. 
After projecting $m$ onto the compressible part by $v = \Lambda^{-1} \div m$ 
as before, 
we first treat the homogeneous system of the first two equations of System \eqref{lCNSK.qvw}~:
\begin{equation}\label{lCNSK.qv}
\left\{
\begin{aligned}
&\partial_t a +  \Lambda v = 0 , \\
&\partial_t v - \underline\nu \Delta v - 
	\gamma \Lambda a - \kappa\rho_* \Lambda^3 a = 0 \\
\end{aligned}
\right.
\end{equation}
where $\gamma \ge0.$ 
Taking the Fourier transform, System \eqref{lCNSK.qv} is transformed into 
\begin{equation}\nonumber	\label{lCNSK.Fqv.mtrx}
\begin{aligned}
&\frac{d}{dt}
  \left[ \begin{array}{c} 
          \widehat a \\
          \widehat v \\
  \end{array} \right] 
= \widehat A(\xi)  \left[ \begin{array}{c} 
          \widehat a \\
          \widehat v \\
  \end{array} \right] 
\quad\text{with}\quad 
\widehat A(\xi):=\left[
\begin{array}{cc}
0 & -|\xi| \\
(\gamma  + \kappa\rho_* |\xi|^2)|\xi| & - \underline\nu |\xi|^2 \\
\end{array}
\right].
\end{aligned}
\end{equation}
The characteristic equation for the matrix $A(\xi)$ is 
$$
X^2+ \underline\nu |\xi|^2X+(\gamma  + \kappa\rho_* |\xi|^2)|\xi|^2= 0,
$$ 
and $\widehat A(\xi)$ has two distinct eigenvalues given by: 
$$\begin{aligned}
\lambda_{\pm}
&=\lambda_{\pm}(\xi) 
:= -\frac{\underline\nu}{2}|\xi|^2
	\biggl(1\pm \sqrt{1-\frac{4(\gamma|\xi| + \kappa\rho_* |\xi|^3)|\xi|}{\underline\nu^2|\xi|^4}}\biggr) \\
&= -\frac{\underline\nu}{2}|\xi|^2
	\biggl(1\pm \sqrt{1-\frac{4\kappa\rho_*}{\underline\nu^2}
	         -\frac{4\gamma}{\underline\nu^2|\xi|^2}}\biggr)\cdotp
\end{aligned}$$ 
Recall that $\gamma\ge0.$ 
From the above formula, we readily gather the following: 
\begin{enumerate}[$(1)$]
\item If 
$\nu^2<4\kappa\rho_*^3,$ then both $\lambda_{\pm}(\xi)$
 are two complex conjugated eigenvalues
 $$
\lambda_{\pm}(\xi) = -\frac{\underline\nu}{2}|\xi|^2\biggl(1\pm i S(|\xi|)\biggr) 
\quad\text{with}\quad \tilde S(|\xi|) = \sqrt{\frac{4\kappa\rho_*}{\underline\nu^2}-1
	         +\frac{4\gamma}{\underline\nu^2|\xi|^2}},
$$
regardless of the choice of $\gamma\ge0,$ for all frequencies. 
If $\gamma>0,$ we have 
$$
\lambda_{\pm}(\xi) 
= -\frac{\underline\nu}{2}|\xi|^2\biggl(1\pm i S(|\xi|)\biggr) 
= -\frac{\underline\nu}{2}|\xi|^2 \mp i \sqrt{\gamma} |\xi| + i O(|\xi|^3)
$$
as $|\xi| \to 0$ and 
$$
\lambda_{\pm}(\xi) 
= -\frac{\underline\nu}{2}|\xi|^2\biggl(1\pm i S(|\xi|)\biggr) 
= -\frac{\underline\nu}{2}|\xi|^2 \mp i \frac{\underline\nu}{2} 
 \left(\frac{4\kappa\rho_*}{\underline\nu^2} -1 \right)^{\frac{1}{2}} |\xi|^2 + i O(1)
$$
as $|\xi| \to \infty.$
If $\gamma=0,$ we have 
$$
\lambda_{\pm}(\xi) = -\frac{\underline\nu}{2}
\biggl(1\pm i\sqrt{\frac{4\kappa\rho_*}{\underline\nu^2}-1} \biggr) |\xi|^2. 
$$
Thus, we may expect parabolic smoothing for both components throughout 
all frequencies. 
\item If $\nu^2=4\kappa\rho_*^3$ and $\gamma>0$, then 
both $\lambda_{\pm}(\xi)$ are two complex conjugated eigenvalues
 $$
\lambda_{\pm}(\xi) = -\frac{\underline\nu}{2}|\xi|^2\biggl(1\pm i \sqrt{\frac{4\gamma}{\underline\nu^2|\xi|^2}}\biggr) 
= -\frac{\underline\nu}{2}|\xi|^2\biggl(1\pm i\frac{2\sqrt{\gamma}}
{\underline\nu|\xi|}\biggr) 
= -\frac{\underline\nu}{2}|\xi|^2 \pm i\sqrt{\gamma} |\xi| 
$$
for all frequencies. 
If $\nu^2=4\kappa\rho_*^3$ and $\gamma=0$, then 
we have a (real) double root 
 $$
\lambda_{\pm}(\xi) = -\frac{\underline\nu}{2}|\xi|^2.
$$
\item 
If $\nu^2>4\kappa\rho_*^3$ and $\gamma>0$ 
then we have two regions of distinct behaviors.
\begin{enumerate}[$(i)$]
\item 
In the low frequency ($|\xi|<\frac{4\gamma\rho_*^2}{\nu^2-4\kappa\rho_*^3}$), 
$\lambda_{\pm}(\xi)$ are two complex conjugated eigenvalues 
$$
\lambda_{\pm}(\xi) = -\frac{\underline\nu}{2}|\xi|^2\biggl(1\pm iS(|\xi|)\biggr) 
\quad\text{with}\quad S(|\xi|) = \sqrt{\frac{4\gamma}{\underline\nu^2|\xi|^2}
			-\left(1-\frac{4\kappa\rho_*}{\underline\nu^2}\right)}
$$
whose real parts are negative and 
\begin{equation}\nonumber\begin{aligned}
\lambda_{\pm}(\xi)
&= -\frac{\underline\nu}{2}|\xi|^2
\biggl(1\pm i
\left( \frac{2\rho_*\sqrt{\gamma}}{\nu |\xi|} - 
	 \frac {\nu^2-4\kappa\rho_*^3} {4 \rho_* \nu\sqrt{\gamma}} |\xi|
  + \mathcal{O} (|\xi|^{3}) \right)
\biggr) \\
&= -\frac{\underline\nu}{2}|\xi|^2 \mp i \sqrt{\gamma} |\xi| 
  + i \mathcal{O} (|\xi|^{3}) 
\quad\text{as}\quad |\xi|\to0.
\end{aligned}\end{equation}
\item On the other hand, in the high frequency ($|\xi|>\frac{4\gamma\rho_*^2}{\nu^2-4\kappa\rho_*^3}$), 
the two real eigenvalues can be expressed as 
$$
\lambda_{\pm}(\xi) = -\frac{\underline\nu}{2}|\xi|^2\biggl(1\pm R(|\xi|)\biggr) 
\quad\text{with}\quad R(|\xi|) = \sqrt{1-\frac{4\kappa\rho_*}{\underline\nu^2}
	         -\frac{4\gamma}{\underline\nu^2|\xi|^2}}
$$
whose real parts are negative and 
\begin{equation}\nonumber
\lambda_{\pm}(\xi) 
= -\frac{\underline\nu}{2}|\xi|^2
\left( 1\pm \left(1-\frac{4\kappa\rho_*^3}{\nu^2} \right)^{\frac{1}{2}} \right) 
+ \mathcal{O}(1) 
\quad\text{as}\quad |\xi|\to\infty
\end{equation}
\end{enumerate}  
If $\nu^2>4\kappa\rho_*^3$ and $\gamma=0$, 
then both $\lambda_{\pm}(\xi)$ are real and purely parabolic:
$$
\lambda_{\pm}(\xi) = -\frac{\underline\nu}{2}\biggl(1\pm  \sqrt{1-\frac{4\kappa\rho_*}{\underline\nu^2} }\biggr) |\xi|^2
$$
\end{enumerate}
We observe that there are two striking differences 
from the barotropic compressible fluid: One is that the the system 
is linearly stable even when the sound speed is zero, i.e., $\gamma=0$. 
Another is that the solutions in \emph{all} frequency modes likely exhibit 
parabolic smoothing effect. 
These two features directly stem from the presence of the Korteweg tensor. 

According to the above ansatz, we seek to derive the natural regularities for 
$(\widehat a, \widehat v)$ by the construction of the 
Lyapunov function in the Fourier space. 
From \eqref{lCNSK.Fqv.mtrx}, we immediately obtain the following sets of equalities:
\begin{equation}
\label{pr-L^2est-eq:a}
\frac{1}{2} \frac{d}{dt} |\widehat a|^2 + |\xi| \Re( \widehat v | \widehat a) = 0, 
\end{equation}
\begin{equation}
\label{pr-L^2est-eq:v}
\frac{1}{2} \frac{d}{dt} |\widehat v|^2 + \underline\nu |\xi|^2 |\widehat v|^2 
- (\gamma +\kappa \rho_* |\xi|^2)|\xi|\Re( \widehat v | \widehat a)= 0,
\end{equation}
and
\begin{equation}
\label{pr-L^2est-eq:av}
\frac{d}{dt} \Re (\widehat a | \widehat v) + |\xi| |\widehat v|^2 - 
 (\gamma +\kappa \rho_* |\xi|^2)|\xi| |\widehat a|^2 +
   \underline\nu |\xi|^2 \Re( \widehat v | \widehat a) = 0,
\end{equation}
where $\Re z$ denotes the real part of a complex number $z$ and $(\cdot |\cdot  )$ is the inner product in ${\mathbb C}$
(that is $(z_1|z_2):=z_1\,\bar z_2$ with $\bar z$ denoting the complex conjugate of $z$). 

\medbreak
The proof below holds for $\gamma\ge0.$
We fix some real parameter $\eta>0.$
By adding $(\gamma + \kappa\rho_*|\xi|^2) \times$\eqref{pr-L^2est-eq:a}, 
\eqref{pr-L^2est-eq:v} and $-\eta |\xi|\times$\eqref{pr-L^2est-eq:av}, 
we discover that 
\begin{equation}\label{pr-L^2est-eq:Lyp}
\begin{aligned}
\frac{1}{2}\frac{d}{dt} &
\left((\gamma+\kappa\rho_*|\xi|^2)|\widehat a|^2 + |\widehat v|^2
	-2\eta |\xi|\Re (\widehat a | \widehat v) \right) \\
&+ (\underline\nu - \eta) |\xi|^2 |\widehat v|^2 + \eta (\gamma+\kappa \rho_* |\xi|^2) |\xi|^2 |\widehat a|^2 - \eta \underline\nu|\xi|^3 \Re (\widehat a | \widehat v)= 0.
\end{aligned}\end{equation}
By Young's inequality, for all $\ep>0$ we obtain the following 
lower bound for the last three terms of the above:
\begin{equation}\nonumber\begin{aligned}
(\underline\nu - \eta) &|\xi|^2 |\widehat v|^2 
	+ \eta (\gamma+\kappa \rho_* |\xi|^2)|\xi|^2 |\widehat a|^2 
	- \eta \underline\nu|\xi|^3 \Re (\widehat a | \widehat v) \\
&\ge \left(\underline\nu - \eta \left(1+\frac{\underline\nu}{\ep} \right)\right) |\xi|^2 |\widehat v|^2 
	+ \eta (\gamma+(\kappa \rho_*-\ep\underline\nu) |\xi|^2)|\xi|^2 |\widehat a|^2.
\end{aligned}\end{equation}
\emph{for all frequencies}.
Taking $\ep = \frac{\kappa\rho_*^3}{2\nu},$ we have 
$
\kappa \rho_*-\ep\underline\nu = \frac{\kappa \rho_*}{2}.
$
Hence, taking $\eta$ such that 
\begin{equation}\nonumber
\eta < \left(\frac{\rho_*}{\nu} + \frac{2\nu}{\kappa\rho_*^3} \right)^{-1},
\end{equation}
we obtain 
\begin{equation}\nonumber\begin{aligned}
(\underline\nu - \eta) |\xi|^2 |\widehat v|^2 &
	+ \eta (\gamma+\kappa \rho_* |\xi|^2)|\xi|^2 |\widehat a|^2 
	- \eta \underline\nu|\xi|^3 \Re (\widehat a | \widehat v) \\
&\ge C_0(\nu,\kappa,\rho_*) |\xi|^2 \left( (\gamma+ |\xi|^2) |\widehat a|^2+ |\widehat v|^2 \right)
\end{aligned}\end{equation}
for some constant $C_0(\nu,\kappa,\rho_*)$ depending only on $\nu,$ $\kappa$ and $\rho_*.$
In the same way, there exists a positive constant $C_1(\nu,\kappa,\rho_*)$ such that 
\begin{equation}\label{pr-L^2est-eq:Lyp.lbg}\begin{aligned}
 C_1(\nu,\kappa,\rho_*)^{-1}\left( (\gamma+ |\xi|^2) |\widehat a|^2+ |\widehat v|^2 \right) 
&\le(\gamma+\kappa\rho_*|\xi|^2)|\widehat a|^2 + |\widehat v|^2
	-2\eta |\xi|\Re (\widehat a | \widehat v) \\
&\le C_1(\nu,\kappa,\rho_*) \left( (\gamma+ |\xi|^2) |\widehat a|^2+ |\widehat v|^2 \right).
\end{aligned}\end{equation}
Indeed, the upper bound is trivial;
In order to prove the lower bound, we use Young's inequality to show
\begin{equation}\nonumber
(\gamma+\kappa\rho_*|\xi|^2)|\widehat a|^2 + |\widehat v|^2
	-2\eta |\xi|\Re (\widehat a | \widehat v) 
\ge \gamma |\widehat a|^2 + (\kappa\rho_* - 2\eta \ep' ) |\xi|^2|\widehat a|^2 
	+ \left(1- \frac{2\eta}{\ep'}\right)|\widehat v|^2
\end{equation}
for all $\ep'>0.$
Therefore, taking $\ep' = \frac{\kappa\rho_*}{4\eta}$ and $\eta \le \frac{2\sqrt{2}}{\kappa\rho_*}$
for example, we obtain the lower bound, and thus, \eqref{pr-L^2est-eq:Lyp.lbg}. 
The above computations combined with \eqref{pr-L^2est-eq:Lyp} imply
$$
\frac{1}{2}\frac{d}{dt} \mathcal{L}^2 + |\xi|^2 |(\gamma\widehat a, |\xi|\widehat a, \widehat v)|^2 \le 0 
\quad\text{with}\quad 
\mathcal{L}^2 := (\gamma+\kappa\rho_*|\xi|^2)|\widehat a|^2 + |\widehat v|^2
	-2\eta |\xi|\Re (\widehat a | \widehat v).
$$
By \eqref{pr-L^2est-eq:Lyp.lbg}, there exists a constant $\tilde c>0$ 
independent of $|\xi|$ such that 
$$
\frac{1}{2}\frac{d}{dt} \mathcal{L}^2 + \tilde c |\xi|^2 \mathcal{L}^2 \le 0,
$$
which implies the exponential stability 
\begin{equation}\label{pr-L^2est-eq:exp.stb}
\mathcal{L}^2(t) \le e^{-\tilde c|\xi|^2 t} \mathcal{L}^2(0),
\end{equation}
for all $t\ge0$. This combined with \eqref{pr-L^2est-eq:Lyp.lbg} again gives 
$$
|(\gamma\widehat a, |\xi|\widehat a, \widehat v)|(t) \le e^{-c_0|\xi|^2 t} |(\gamma\widehat a, |\xi|\widehat a, \widehat v)|(0).
$$
\medbreak
Granted with the above estimate in the frequency space, 
it is now straightforward to complete the proof of the lemma~: 
it suffices to multiply the Littlewood-Paley decomposition $\widehat \phi_j$, 
perform exactly the same computations as above on 
$$(\gamma\widehat\phi_j\widehat a, 
	2^{j} \widehat\phi_j \widehat a, \widehat\phi_j \widehat v)$$ 
to obtain a bound similar to \eqref{pr-L^2est-eq:exp.stb}:
\begin{equation}\nonumber
\mathcal{L}^2(t) \le e^{-\tilde c 2^{2j} t} \mathcal{L}^2(0), 
\end{equation}
with 
$$
\mathcal{L}^2 := (\gamma+\kappa\rho_*2^{2j})| \widehat\phi_j\widehat a|^2 
			+ |\widehat\phi_j\widehat v|^2
			-2\eta 2^j \Re (\widehat\phi_j\widehat a | \widehat\phi_j\widehat v).
$$
Similarly to \eqref{pr-L^2est-eq:Lyp.lbg}, we have 
\begin{equation}\nonumber
\mathcal{L}^2 \cong |(\gamma \widehat\phi_j\widehat a, 2^j  \widehat\phi_j \widehat a,  \widehat\phi_j\widehat v)|,
\end{equation}
so we have obtained 
\begin{equation}\nonumber
 |(\gamma \widehat\phi_j\widehat a, 2^j  \widehat\phi_j \widehat a,  \widehat\phi_j\widehat v)|(t) 
 \le e^{-\tilde c 2^{2j} t}  |(\gamma \widehat\phi_j\widehat a, 2^j  \widehat\phi_j \widehat a, \widehat\phi_j \widehat v)|(0). 
\end{equation}
Then taking $L^{2}(\R^d)$ norms of both sides of the resulting inequality, 
we obtain 
\begin{equation}\nonumber
\|(\gamma \widehat\phi_j\widehat a, 2^j  \widehat\phi_j \widehat a,  \widehat\phi_j\widehat v)\|_{L^{2}} 
\le e^{-c_02^{2j} t} \|(\gamma \widehat\phi_j\widehat a_0, 2^j \widehat\phi_j  \widehat a_0,  \widehat\phi_j\widehat v_0)\|_{L^{2}}.
\end{equation}
For $\gamma>0$, this implies that for any integer $m$ there exist constants depending only on $m$ 
such that the action of the semigroup $e^{t\widehat A}$ is given by 
\begin{equation}\label{l:lFH.FLp.l}
\|e^{t\widehat A}(\gamma \widehat\phi_j \widehat a_0, \widehat\phi_j \widehat v_0)\|_{L^{2}} 
\le C e^{-c_02^{2j} t} \|(\gamma \widehat\phi_j\widehat a_0, \widehat\phi_j \widehat v_0)\|_{L^{2}} 
\quad\text{for}\quad j \le m
\end{equation}
and
\begin{equation}\label{l:lFH.FLp.h}
\|e^{t\widehat A}(2^j  \widehat\phi_j \widehat a_0,  \widehat\phi_j\widehat v_0)\|_{L^{2}} 
\le C e^{-c_02^{2j} t} \|(2^j  \widehat\phi_j \widehat a_0,  \widehat\phi_j\widehat v_0)\|_{L^{2}}
\quad\text{for}\quad  m<j,
\end{equation}
where $(\widehat a_0, \widehat v_0) =(\widehat a, \widehat v) (0).$ 
Thanks to \eqref{l:lFH.FLp.l} and \eqref{l:lFH.FLp.h}, we deduce 
\begin{equation}\label{l:lFH.FLp.smgrp}
\|e^{t\widehat A}(\gamma\widehat\phi_j  \widehat a_0, 2^j \widehat\phi_j  \widehat a_0,\widehat\phi_j  \widehat v_0)\|_{L^{2}} 
\le C e^{-c_02^{2j} t} 
\|(\gamma\widehat\phi_j  \widehat a_0, 2^j \widehat\phi_j  \widehat a_0, \widehat\phi_j \widehat v_0)\|_{L^{2}}
\quad\text{for all}\quad j \in\Z.
\end{equation}

\medbreak
To handle the source term, splitting the frequency regions into low and high and resorting to Duhamel's formula with the aid of \eqref{l:lFH.FLp.l} 
and \eqref{l:lFH.FLp.h}, we obtain
\begin{equation}\nonumber\begin{aligned}
\|(\gamma\widehat\phi_j\widehat a, 2^j \widehat\phi_j \widehat a,\widehat\phi_j \widehat v)\|_{L^{p'}}
&\lesssim e^{-c_02^{2j} t} \|(\gamma\widehat\phi_j\widehat a_0, 2^j \widehat\phi_j \widehat a_0, \widehat\phi_j\widehat v_0)\|_{L^{2}} \\
&\qquad+  \int_0^{t} 
e^{-c_02^{2j} (t-\tau)} 
\|(\gamma\widehat\phi_j\widehat f, 2^j \widehat\phi_j \widehat f,\widehat\phi_j \widehat g)\|_{L^{2}} d\tau. 
\end{aligned}\end{equation}
Then, time-integration along with the convolution inequality leads to 
a frequency-localized time-space estimate. 
Summing up the resulting inequality over $j \in \Z$ with a weight $2^{js}$, 
we obtain the estimate for the compressible part. 
Since the incompressible component of the velocity 
$w :=(\id -\Lambda^{-1} \div )u$ satisfies a mere heat equation 
$$
\partial_t w - \underline\mu \Delta w =(\id -\Lambda^{-1} \div )g, 
$$
the standard parabolic estimate (Lemma \ref{l:heatbesov}) for $w$ yields the desired estimates for $w$ as well. 
\end{proof}

As a corollary of the proof of Lemma \ref{l:lFH}, we may deduce the following 
estimate which is used to establish decay estimates. 
\begin{lem}\label{l:lFH.decay}
Let $d\ge2$, $s_1, s_2 \in\R$ and $s_1>s_2$. 
Let $e^{tK}$ be the semigroup associated to 
System \eqref{lCNSKp} with $f\equiv g \equiv 0$ and 
initial data $(a_0,m_0) \in \dot B^{s_2}_{2,\infty}(\R^d)$. 
Then there exists a constant $C$ depending on $d$, 
$s_1$ and $s_2$ such that the following estimate holds:
\begin{equation}\nonumber
\|e^{tK} ((\gamma + \Lambda)a_0,m_0)\|_{\dot B^{s_1}_{2,1}} \le 
C t^{-\frac{s_1-s_2}{2}}\|((\gamma + \Lambda)a_0,m_0)\|_{\dot B^{s_2}_{2,\infty}} \quad\text{for all} \quad t>0. 
\end{equation}
\end{lem}
\begin{proof}
From \eqref{l:lFH.FLp.smgrp} and the corresponding estimate for $w$, we 
readily deduce  
\begin{equation}\nonumber
\|e^{t K}(  \Delta_j (\gamma + \Lambda) a_0,\Delta_j m_0)\|_{L^{2}} 
\le C e^{-c_02^{2j} t} 
\|(\Delta_j (\gamma + \Lambda) a_0, \Delta_j m_0)\|_{L^{2}}
\end{equation}
for all $j\in\Z$. 
By Lemma \ref{lem:unifbd} with $r\equiv 1$, we readily have 
\begin{equation}\nonumber\begin{aligned}
&\left\| \left\{ 2^{js_1} \|e^{t K}(  \Delta_j (\gamma + \Lambda) a_0,\Delta_j m_0)\|_{L^{2}} 
\right\}_{j\in\Z} \right\|_{\ell^{1}} \\
&= C t^{-\frac{s_1-s_2}{2}} \left\| \left\{ (t^{\frac{1}{2}} 2^{j})^{(s_1-s_2)} e^{-c_02^{2j} t}  \right\}_{j\in\Z}
\right\|_{\ell^{1}} \|((\gamma + \Lambda)a_0,m_0)\|_{\dot B^{s_2}_{2,\infty}} \\
& \le C t^{-\frac{s_1-s_2}{2}}\|((\gamma + \Lambda)a_0,m_0)\|_{\dot B^{s_2}_{2,\infty}}.
\end{aligned}\end{equation}
\end{proof}

\section{Nonlinear problem}
In this section, we show the existence and uniqueness of 
global solution to the nonlinear problem \eqref{CNSKp:rho.m}. 

\smallbreak
\subsection{Banach's fixed point theorem}
We first prove the following auxiliary lemma, 
which is a generalization of Lemma 2.1 in \cite{CheGal2009} to bi-
and-trilinear operators. 
\begin{lem}\label{l:fxd.pt.bi.tri.lin}
Let $(X, \|\cdot\|_X)$ be a Banach space. Let $B:X \times X \to X$
be a bilinear continuous operator with norm $K_2$
and $T:X\times X \times X \to X$ be a trilinear continuous operators with norm $K_3$. 
Let further $L :X\to X$ be a continuous linear operator with norm $N<1.$ 
Then for all $y\in X$ such that 
$$
\|y\|_{X} <  \min \left\{\frac{1-N}{2}, \frac{(1-N)^2}{2(2K_2+ 3 K_3)} \right\}, 
$$
Equation $x = y+ L(x) + B(x,x) + T(x,x, x)$ 
has a unique solution $x$ in the ball $B^X_{\tilde R}(0)$ 
of center 0 and radius $\tilde R:=\min \left\{1, \frac{1-N}{2K_2+ 3 K_3 } \right\}.$ 
In addition, $x$ satisfies 
\begin{equation}\label{l:fxd.pt.bi.tri.lin.bnd}
\|x\|_{X} \le \frac{2}{1-N} \|y\|_X. 
\end{equation}
\end{lem}
\begin{proof}
Although this is standard, we prove the lemma for completeness. 
Letting $\Phi : x \mapsto y+ L(x)+ B(x,x) + T(x,x,x),$ 
we prove that $\Phi$ is a contraction map in some ball.
Let $I:=\|y\|_{X}$ and let $x$ be in some closed ball $\overline{B^X_{R}(0)}$ of $X$.
Assuming $R\le 1,$ we have 
\begin{equation}\label{l:fxd.pt.bi.tri.lin:pr1}
\begin{aligned}
\|\Phi(x)\|_{X} &\le \|y\|_{X} +\|L(x)\|_{X}+\|B(x,x)\|_X +\|T(x,x,x)\|_{X} \\
 &\le I+ N R + (K_2 + K_3 )R^2.
\end{aligned}
\end{equation}
Let $R$ be defined by 
\begin{equation}\label{l:fxd.pt.bi.tri.lin:d.R}
R := \frac{2}{1-N} I. 
\end{equation}
Then we have from \eqref{l:fxd.pt.bi.tri.lin:pr1}
\begin{equation}\nonumber
\|\Phi(x)\|_{X} 
\le \left( \frac{1-N}{2} + N + (K_2 + K_3) R \right) R.
\end{equation}
Thus, suffices to take $R$ such that  
\begin{equation}\label{l:fxd.pt.bi.tri.lin:stab}
R \le \min \left\{ 1, \frac{1-N}{2(K_2 + K_3)} \right\}.
\end{equation}
for $\Phi$ to be a self-map on the closed ball $\overline{B^X_R(0)}.$ 

\smallbreak
On the other hand, 
we have for any $x_1, x_2 \in \overline{B^X_R(0)}$ and 
for $R\le 1$,
$$
\begin{aligned}
\|\Phi&(x_1)-\Phi(x_2)\|_{X} 
\le \|L(x_1-x_2)\|_X  
+ \|B(x_1-x_2, x_1)\|_{X}+ \|B(x_1, x_1-x_2)\|_{X} \\
&\quad + \|T(x_1-x_2, x_1, x_1)\|_{X} 
+ \|T(x_1, x_1-x_2, x_1)\|_{X} 
+ \|T(x_2, x_2, x_1-x_2)\|_{X} \\
 &\le (N + 2K_2R + 3 K_3 R^2) \|x_1-x_2\|_{X} 
 \le (N + (2K_2+ 3 K_3 )R) \|x_1-x_2\|_{X}. 
\end{aligned}
$$ 
Hence, in order for $\Phi$ to be a contraction, 
$N + (2K_2+ 3 K_3 )R$ has to hold; this amounts to taking $R$ such that
\begin{equation}\label{l:fxd.pt.bi.tri.lin:cntr}
R < \min \left\{1, \frac{1-N}{2K_2+ 3 K_3 } \right\}.
\end{equation}
Thus, from \eqref{l:fxd.pt.bi.tri.lin:d.R}, 
\eqref{l:fxd.pt.bi.tri.lin:stab} and \eqref{l:fxd.pt.bi.tri.lin:cntr}, 
Banach's fixed point theorem is applicable for all $y \in X$ such that 
$$
 \frac{2}{1-N} \|y\|_{X} < \min \left\{1, \frac{1-N}{2K_2+ 3 K_3 } \right\}, 
\ \text{i.e.,} \ 
\|y\|_{X} <  \min \left\{\frac{1-N}{2}, \frac{(1-N)^2}{2(2K_2+ 3 K_3)} \right\}, 
$$
The uniqueness of the fixed point in the ball naturally follows. 
Routine computations as in the stability estimate 
lead to \eqref{l:fxd.pt.bi.tri.lin.bnd}. 
\end{proof}

\smallbreak
\subsection{Global results in $L^2(\R^d)$-based spaces when $\gamma>0$}
Lemma \ref{l:lFH} yields the following. 
\begin{cor}\label{c:lL2}
Let $d\ge2$, $T>0$, $s\in\R$, $\gamma\ge0$ 
and 
$$
((\gamma+\Lambda) f, g) 
	\in \widetilde{L^1}(0,T; \dot{B}^{s}_{2,\sigma}(\R^d)).
$$
Assume that $(a,m)$ is a solution for \eqref{lCNSKp} such that 
\begin{equation}\nonumber
((\gamma+\Lambda) a, m) \in \widetilde{L^\infty}(0,T; \dot{B}^s_{2,\sigma}(\R^d))
\cap \widetilde{L^1}(0,T;\dot{B}^{s+2}_{2,\sigma}(\R^d)).
\end{equation} 
Then there exists some positive constant $C$ depending only on $\rho_*,$ $P$, 
$\lambda$, $\mu$, $\kappa$ and $d$
such that the following estimate hold : if $\gamma\ge0,$ then
\begin{equation}\nonumber
\begin{aligned}
\|((\gamma+\Lambda) a, m)&\|_{\widetilde{L^\infty_t}(\dot{B}^s_{2,\sigma})}
+ \|\Lambda^2 ((\gamma+\Lambda) a, m)\|_{\widetilde{L^1_t}(\dot{B}^{s}_{2,\sigma})} \\
& \le C \big( \|((\gamma+\Lambda) a_0, m_0)\|_{\dot{B}^{s}_{2,\sigma}} 
+\|((\gamma+\Lambda) f, g)\|_{\widetilde{L^1_t}(\dot{B}^{s}_{2,\sigma})} \big)
\end{aligned}\end{equation}
for all $0 <t \le T.$
\end{cor}

Given initial data $(a_0,m_0)$ such that 
$((\gamma+\Lambda) a_0, m_0) \in \dot{B}^{\frac{d}{2}-1}_{2,1}(\R^d)$, 
in order to solve \eqref{CNSKp:rho.m}, it suffices to show that the map
\begin{equation}\label{CNSK.map}
\Phi : (b,n) \mapsto (a,m)
\end{equation}
with $(a,  m)$ the solution to 
\begin{equation}\label{CNSKp.scheme}
\left\{\begin{aligned}
&\partial_t  a + \div  m = 0 ,\\
&\partial_t  m - \frac{1}{\rho_*} \mathcal{L}  m  
    -\kappa \rho_* \nabla\Delta  a + \gamma \nabla  a = N(b,n), \\
&(a,  m)_{t=0} =(a_0,m_0),
\end{aligned} \right.
\end{equation}
where 
\begin{equation}\label{CNSKp:mom.nonl.trm}
\begin{aligned}
N(b,n) &:= \mathcal{L}\left( Q(b) n\right) 
    +  \kappa b \nabla\Delta b 
     - \frac{1}{\rho_*}  \div(n\otimes n) \\
&\qquad
    -\div\left( Q(b)n\otimes n \right) 
    +\left( P'(b+\rho_*)- \gamma \right) \nabla b,
\end{aligned}\end{equation}
has a fixed point in a ball of a suitable metric space. 

We define a Banach space $CL_T$ by 
\begin{equation}\nonumber
CL_T := \widetilde{L^\infty}(0,T;\dot B^{\frac{d}{2}-1}_{2,1}(\R^d)) 
		\cap \widetilde{L^1}(0,T;\dot B^{\frac{d}{2}+1}_{2,1}(\R^d))
\end{equation}
equipped with a norm 
\begin{equation}\nonumber
\|(a,m)\|_{CL_T} := \|((\gamma+\Lambda) a, m)\|_{\widetilde{L^\infty_T}(\dot B^{\frac{d}{2}-1}_{2,1})} 
		+ \|\Lambda^2 ((\gamma+\Lambda) a, m)\|_{\widetilde{L^1_T}(\dot B^{\frac{d}{2}-1}_{2,1})}.
\end{equation}
By interpolation, we may readily see that $(a,m) \in CL_T$ further satisfies 
\begin{equation}\label{CLT.intp}
((\gamma+\Lambda) a, m) \in 
\widetilde{L^q}(0,T;\dot B^{\frac{d}{2}-1+\frac{2}{q}}_{2,1}(\R^d))
\quad\text{for}\quad q \in [1,\infty]. 
\end{equation}

Recall that $N$ defined in \eqref{CNSKp:mom.nonl.trm} is 
a combination of bi-and-trilinear terms. We have the following nonlinear estimates. 
\begin{lem}\label{l:CNSK.nonl.est}
Let $d\ge 2$ and $T>0$. 
Let $N$ be as defined in \eqref{CNSKp:mom.nonl.trm} and 
let $P$ be a smooth function such that $\gamma\ge0$. 
Then there exists a positive constant $C$ depending only on $\rho_*,$ $P$, 
$\lambda$, $\mu$, $\kappa$ and $d$ such that 
\begin{equation}\label{l:CNSK.nonl.est:inq1}
\|N(b,n)\|_{\widetilde{L^1_t}(\dot{B}^{\frac{d}{2}-1}_{2,1})} 
\le C (\|(b,n)\|_{CL_t}^2+\|(b,n)\|_{CL_t}^3)
\end{equation}
for all $(b,n) \in CL_T$, $0 \le t \le T$, provided further that $\gamma>0$.

Furthermore, there exists a positive constant $C$
depending only on $\rho_*$, $P$, 
$\lambda$, $\mu$, $\kappa$ and $d$ such that 
\begin{equation}\label{l:CNSK.nonl.est:inq2}
\|N(b,n)\|_{\widetilde{L^1_t}(\dot{B}^{\frac{d}{2}-1}_{2,1})} 
\le C ((1+t)\|(b,n)\|_{CL_t}^2+\|(b,n)\|_{CL_t}^3)
\end{equation}
for all $(b,n) \in CL_T$, $0 \le t \le T$, provided further that $\gamma=0$.
\end{lem}
\begin{proof}
Taking $q=1,2,\infty$ in \eqref{CLT.intp}, we in particular have 
\begin{equation}\label{CLT.intp.pr}
(\gamma b,\Lambda b, n) \in 
 \widetilde{L^\infty}(0,T;\dot B^{\frac{d}{2}-1}_{2,1}(\R^d))
\cap \widetilde{L^2}(0,T;\dot B^{\frac{d}{2}}_{2,1}(\R^d))
\cap \widetilde{L^1}(0,T;\dot B^{\frac{d}{2}+1}_{2,1}(\R^d))
\end{equation}
for all $(b,n) \in CL_T.$ 
Note that $Q$ satisfies 
$Q'(b) = \frac{1}{(b+\rho_*)^2}$ and $Q'(0)=\frac{1}{\rho_*^2}>0$. 
Thus, thanks to Lemma \ref{l:comp} and Lemma \ref{lem:prd.p=2}, 
we have 
\begin{equation}\label{l:CNSK.nonl.est.pr1}\begin{aligned}
\|&\mathcal{L}\left( Q(b) n\right) \|_{\widetilde{L^1_t}(\dot{B}^{\frac{d}{2}-1}_{2,1})} 
\le C \|\nabla( Q(b)) n + Q(b) \div n \id\|_{\widetilde{L^1_t}(\dot{B}^{\frac{d}{2}}_{2,1})} \\
&\le C \|(Q'(b) -Q'(0) )(\nabla b) n \|_{\widetilde{L^1_t}(\dot{B}^{\frac{d}{2}}_{2,1})}
	+ \| Q'(0) (\nabla b) n\|_{\widetilde{L^1_t}(\dot{B}^{\frac{d}{2}}_{2,1})} 
	+ \|Q(b) \div n \|_{\widetilde{L^1_t}(\dot{B}^{\frac{d}{2}}_{2,1})} \\
&\le C \bigg( \|b\|_{\widetilde{L^\infty_t}(\dot{B}^{\frac{d}{2}}_{2,1})} 
	\|\nabla b \|_{\widetilde{L^2_t}(\dot{B}^{\frac{d}{2}}_{2,1})} 
	\|n\|_{\widetilde{L^2_t}(\dot{B}^{\frac{d}{2}}_{2,1})}
+  \|\nabla b\|_{\widetilde{L^2_t}(\dot{B}^{\frac{d}{2}}_{2,1})} 
	\|n\|_{\widetilde{L^2_t}(\dot{B}^{\frac{d}{2}}_{2,1})}\\
&\qquad\qquad  + \|b \|_{\widetilde{L^\infty_t}(\dot{B}^{\frac{d}{2}}_{2,1})}
		\|\div n \|_{\widetilde{L^1_t}(\dot{B}^{\frac{d}{2}}_{2,1})} \bigg).
\end{aligned}\end{equation}
Now, a standard interpolation argument ensures that 
the right hand-side of \eqref{l:CNSK.nonl.est.pr1} is bounded by 
\begin{equation}\nonumber
\|(b,n)\|_{CL_T}^2+\|(b,n)\|_{CL_T}^3.
\end{equation}
Similarly, using \eqref{CLT.intp.pr}, we have 
\begin{equation}\label{l:CNSK.nonl.est.pr2}
\|\kappa b \nabla\Delta b \|_{\widetilde{L^1_t}(\dot{B}^{\frac{d}{2}-1}_{2,1})} 
\le C \|b \|_{\widetilde{L^\infty_t}(\dot{B}^{\frac{d}{2}}_{2,1})} 
	\|\nabla\Delta b \|_{\widetilde{L^1_t}(\dot{B}^{\frac{d}{2}-1}_{2,1})}
\le C \|b \|_{\widetilde{L^\infty_t}(\dot{B}^{\frac{d}{2}}_{2,1})} 
	\|\Delta b \|_{\widetilde{L^1_t}(\dot{B}^{\frac{d}{2}}_{2,1})};
\end{equation}
\begin{equation}\label{l:CNSK.nonl.est.pr3}
\left\|\frac{1}{\rho_*}  \div(n\otimes n) \right\|_{\widetilde{L^1_t}(\dot{B}^{\frac{d}{2}-1}_{2,1})} 
\le C \| n\otimes n \|_{\widetilde{L^1_t}(\dot{B}^{\frac{d}{2}}_{2,1})} 
\le C \|n\|_{\widetilde{L^2_t}(\dot{B}^{\frac{d}{2}}_{2,1})}^2;
\end{equation}
\begin{equation}\label{l:CNSK.nonl.est.pr4}\begin{aligned}
\|\div\left( Q(b)n\otimes n \right)  \|_{\widetilde{L^1_t}(\dot{B}^{\frac{d}{2}-1}_{2,1})} 
\le C\|Q(b)n\otimes n \|_{\widetilde{L^1_t}(\dot{B}^{\frac{d}{2}}_{2,1})}  
\le C\|b\|_{\widetilde{L^\infty_t}(\dot{B}^{\frac{d}{2}}_{2,1})}
	\|n\|_{\widetilde{L^2_t}(\dot{B}^{\frac{d}{2}}_{2,1})}^2.
\end{aligned}\end{equation}
For the pressure term when $\gamma>0$,  
\begin{equation}\label{l:CNSK.nonl.est.pr5}\begin{aligned}
\|\left( P'(b+\rho_*)- P'(\rho_*) \right) \nabla b \|_{\widetilde{L^1_t}(\dot{B}^{\frac{d}{2}-1}_{2,1})} 
&\le C\| P'(b+\rho_*)- P'(\rho_*) \|_{\widetilde{L^\infty_t}(\dot{B}^{\frac{d}{2}-1}_{2,1})} 
 \|\nabla b \|_{\widetilde{L^1_t}(\dot{B}^{\frac{d}{2}}_{2,1})}  \\
&\le C\| b\|_{\widetilde{L^\infty_t}(\dot{B}^{\frac{d}{2}-1}_{2,1})} 
 \|\nabla b \|_{\widetilde{L^1_t}(\dot{B}^{\frac{d}{2}}_{2,1})},\\
\end{aligned}\end{equation}
where we have used the regularity for $\gamma b$ in \eqref{CLT.intp.pr}.
Thus, combining 
\eqref{l:CNSK.nonl.est.pr1},
\eqref{l:CNSK.nonl.est.pr2},
\eqref{l:CNSK.nonl.est.pr3},
\eqref{l:CNSK.nonl.est.pr4},
and \eqref{l:CNSK.nonl.est.pr5}, 
we have 
\begin{equation}\nonumber\begin{aligned}
\|N(b,n)\|_{\widetilde{L^1_t}(\dot{B}^{\frac{d}{2}-1}_{2,1})}
&\lesssim 
(\|b\|_{\widetilde{L^\infty_t}(\dot{B}^{\frac{d}{2}}_{2,1})} +1)
	\|\nabla b \|_{\widetilde{L^2_t}(\dot{B}^{\frac{d}{2}}_{2,1})} 
	\|n\|_{\widetilde{L^2_t}(\dot{B}^{\frac{d}{2}}_{2,1})}\\
& + \|b \|_{\widetilde{L^\infty_t}(\dot{B}^{\frac{d}{2}}_{2,1})}
	\left(\|\div n \|_{\widetilde{L^1_t}(\dot{B}^{\frac{d}{2}}_{2,1})} 
	+ \|\Delta b \|_{\widetilde{L^1_t}(\dot{B}^{\frac{d}{2}}_{2,1})}\right) \\
&  +\|n\|_{\widetilde{L^2_t}(\dot{B}^{\frac{d}{2}}_{2,1})}^2 
\left(1 +\|b\|_{\widetilde{L^\infty_t}(\dot{B}^{\frac{d}{2}}_{2,1})} \right)
  +\| b\|_{\widetilde{L^\infty_t}(\dot{B}^{\frac{d}{2}-1}_{2,1})} 
	 \|\nabla b \|_{\widetilde{L^1_t}(\dot{B}^{\frac{d}{2}}_{2,1})},
\end{aligned}\end{equation}
provided that $\gamma>0$. 
Thus, we have \eqref{l:CNSK.nonl.est:inq1} if $\gamma>0$. 

\medbreak
If, on the other hand, $\gamma=0$, we only have
\begin{equation}\nonumber
\Lambda b  \in 
\widetilde{L^\infty}(0,T;\dot B^{\frac{d}{2}-1}_{2,1}(\R^d))
\cap \widetilde{L^2}(0,T;\dot B^{\frac{d}{2}}_{2,1}(\R^d))
\cap \widetilde{L^1}(0,T;\dot B^{\frac{d}{2}+1}_{2,1}(\R^d)).
\end{equation}
We may write the term 
$P'(b+\rho_*) \nabla b= \nabla (P(b+\rho_*))$ and 
\begin{equation}\nonumber\begin{aligned}
\|\nabla (P(b+\rho_*))\|_{\widetilde{L^1_t}(\dot{B}^{\frac{d}{2}-1}_{2,1})} 
&\le C \| P(b+\rho_*)\|_{\widetilde{L^1_t}(\dot{B}^{\frac{d}{2}}_{2,1})} \\
&= C \| P(b+\rho_*)- P(\rho_*)\|_{\widetilde{L^1_t}(\dot{B}^{\frac{d}{2}}_{2,1})} 
\le C \| b\|_{\widetilde{L^1_t}(\dot{B}^{\frac{d}{2}}_{2,1})},
\end{aligned}\end{equation}
thanks to Lemma \ref{l:comp}.
However, we have no control on 
$\| b\|_{\widetilde{L^1_t}(\dot{B}^{\frac{d}{2}}_{2,1})},$ so the above only implies
\begin{equation}\nonumber
\|\nabla (P(b+\rho_*))\|_{\widetilde{L^1_t}(\dot{B}^{\frac{d}{2}-1}_{2,1})} 
\le C t \| b\|_{\widetilde{L^\infty_t}(\dot{B}^{\frac{d}{2}}_{2,1})}.
\end{equation}
Alternatively, we have 
\begin{equation}\label{l:CNSK.nonl.est.pr6}\begin{aligned}
\|P'(b+\rho_*) \nabla b \|_{\widetilde{L^1_t}(\dot{B}^{\frac{d}{2}-1}_{2,1})} 
&\le C\| P'(b+\rho_*) \|_{\widetilde{L^1_t}(\dot{B}^{\frac{d}{2}}_{2,1})} 
 \|\nabla b \|_{\widetilde{L^\infty_t}(\dot{B}^{\frac{d}{2}-1}_{2,1})}  \\
&\le C\| b\|_{\widetilde{L^1_t}(\dot{B}^{\frac{d}{2}}_{2,1})} 
 \|\nabla b \|_{\widetilde{L^\infty_t}(\dot{B}^{\frac{d}{2}-1}_{2,1})}\\
&\le C t \| b\|_{\widetilde{L^\infty_t}(\dot{B}^{\frac{d}{2}}_{2,1})} 
 \|\nabla b \|_{\widetilde{L^\infty_t}(\dot{B}^{\frac{d}{2}-1}_{2,1})}
\le C t  \|\nabla b \|_{\widetilde{L^\infty_t}(\dot{B}^{\frac{d}{2}-1}_{2,1})}^2.\\ 
\end{aligned}\end{equation}
Either way, we may not close the estimates globally-in-time if 
$\gamma=0$ with our given regularity in $CL_T$: Combining 
\eqref{l:CNSK.nonl.est.pr1}, \eqref{l:CNSK.nonl.est.pr2},
\eqref{l:CNSK.nonl.est.pr3}, \eqref{l:CNSK.nonl.est.pr4},
and \eqref{l:CNSK.nonl.est.pr6}, we have 
\begin{equation}\nonumber\begin{aligned}
\|N(b,n)\|_{\widetilde{L^1_t}(\dot{B}^{\frac{d}{2}-1}_{2,1})}
&\lesssim 
(\|b\|_{\widetilde{L^\infty_t}(\dot{B}^{\frac{d}{2}}_{2,1})} +1)
	\|\nabla b \|_{\widetilde{L^2_t}(\dot{B}^{\frac{d}{2}}_{2,1})} 
	\|n\|_{\widetilde{L^2_t}(\dot{B}^{\frac{d}{2}}_{2,1})}\\
& + \|b \|_{\widetilde{L^\infty_t}(\dot{B}^{\frac{d}{2}}_{2,1})}
	\left(\|\div n \|_{\widetilde{L^1_t}(\dot{B}^{\frac{d}{2}}_{2,1})} 
	+ \|\Delta b \|_{\widetilde{L^1_t}(\dot{B}^{\frac{d}{2}}_{2,1})}\right) \\
&  +\|n\|_{\widetilde{L^2_t}(\dot{B}^{\frac{d}{2}}_{2,1})}^2 
\left(1 +\|b\|_{\widetilde{L^\infty_t}(\dot{B}^{\frac{d}{2}}_{2,1})} \right)
  + t \| b\|_{\widetilde{L^\infty_t}(\dot{B}^{\frac{d}{2}}_{2,1})}^2
\end{aligned}\end{equation}
in this case. 
Thus, we have \eqref{l:CNSK.nonl.est:inq2} if $\gamma=0$. 
This completes the proof of Lemma \ref{l:CNSK.nonl.est}. 
\end{proof}
\begin{proof}[Proof of Theorem \ref{t:L2.WP}] 
The time-continuity follows from the standard density argument as in 
the corresponding proof for the linear heat equation. 
The proof of Theorem \ref{t:L2.WP} readily follows from 
Corollary \ref{c:lL2}, Lemma \ref{l:CNSK.nonl.est} 
and Lemma \ref{l:fxd.pt.bi.tri.lin}. 
As the proof is standard, we only sketch the steps. 
We let $(a_L,m_L)$ be the solution to the linear equation:
\begin{equation}\nonumber
\left\{\begin{aligned}
&\partial_t  a_L + \div  m_L = 0 ,\\
&\partial_t  m_L - \frac{1}{\rho_*} \mathcal{L}  m_L  
    -\kappa \rho_* \nabla\Delta  a_L
   + \gamma \nabla  a_L = 0, \\
&(a,  m)_{t=0} =(a_0,m_0)
\end{aligned} \right.
\end{equation}
and let $(\tilde a, \tilde m) := (a- a_L, m-m_L).$ 
Then Corollary \ref{c:lL2} with $s\equiv\frac{d}{2}-1$ and $\sigma\equiv1$ gives
\begin{equation}\nonumber
\begin{aligned}
\|(a, m)\|_{CL_T} 
&\le \|(a_L, m_L)\|_{CL_T}  + \|(\tilde a, \tilde m)\|_{CL_T} \\
& \le C 
\big( \|((\gamma+\Lambda) a_0, m_0) \|_{\dot{B}^{\frac{d}{2}-1}_{2,1}} +
\|N(b,n)\|_{\widetilde{L^1_t}(\dot{B}^{\frac{d}{2}-1}_{2,1})} \big),
\end{aligned}\end{equation}
for all $(b,n) \in CL_T$. 
Lemma \ref{l:CNSK.nonl.est} ensures that $N(b,n)$
can be regarded as a combination of bi-and-trilinear continuous 
operators in $CL_T$. 
We define a ball in $CL_T$ centered at the origin by 
\begin{equation}\label{CNSK.mtrc.sp}
B_R^{CL_T}(0) := \left\{ (a,m) \in CL_T \,;\, \|(a, m)\|_{CL_T} \le R \right\},
\end{equation}
where $R>0$. Lemma \ref{l:fxd.pt.bi.tri.lin} now ensures the existence of a 
unique solution in the ball $B_R^{CL_T}(0)$ for a sufficiently small data. 
Note that $T=\infty$ is allowed for in the case $\gamma>0$, 
thanks to the global estimate in Lemma \ref{l:CNSK.nonl.est}. 
By \eqref{l:fxd.pt.bi.tri.lin.bnd}, 
$\|(a, m)\|_{CL_T}  
\le C \|((\gamma+\Lambda) a_0, m_0) \|_{\dot{B}^{\frac{d}{2}-1}_{2,1}}$ 
is also ensured. 
\end{proof}

\subsection{Global solution when $\gamma=0$.} 
From hereon, we focus on the case $\gamma=0$. 
When $\gamma>0$, 
the term $\gamma \nabla \rho$ induces a damping effect, 
which provides the control of $ \|\nabla b \|_{\widetilde{L^1_t}(\dot{B}^{\frac{d}{2}}_{2,1})}.$ 
The lack of this time-integrability in the case $\gamma=0$ is troublesome, 
especially when we control the pressure term. 
As shown by \eqref{l:CNSK.nonl.est.pr6}, we need the control of 
$ \|\nabla b \|_{\widetilde{L^1_t}(\dot{B}^{\frac{d}{2}}_{2,1})}$, 
whereas we only have 
$\nabla b \in \widetilde{L^\infty_t}(\dot{B}^{\frac{d}{2}}_{2,1})$ in $CL_T$. 
In order to compensate this lack of global control, 
we impose stronger regularity on the low-frequency, 
which is of lower order in terms of the scaling, 
to gain the necessary global control of the density. 
\medbreak 

To handle this case, we have to change the expressions of 
the capillary and pressure terms. 
We linearize the capillary term while keeping it in the divergence form:
\begin{equation}\nonumber
\div K(\rho) 
= \kappa \rho_* \nabla \Delta a +  \div \left( K(a) \right),
\end{equation}
where as before we denoted $a:= \rho-\rho_*$. 
Similarly, we write the pressure as follows: 
By Taylor expansion, we have 
$$
P(a+\rho_*)=P(\rho_*)+P'(\rho_*)a+a\widetilde P(a),
$$
where $\widetilde P: \R_+ \to\R$ is smooth and vanishes at $0.$
Using the assumption that $P'(\rho_*) = 0$, we have 
\begin{equation}\nonumber
\nabla P(a+\rho_*) =  \nabla ( a\widetilde P(a) ).
\end{equation}
With this expression, we obtain the same linear structure with 
nonlinear terms all written in divergence form:
\begin{equation}\label{CNSKp:mom.nonl.trm.div}
\begin{aligned}
N(b,n)= \mathcal{L}\left( Q(b) n\right) 
    + \div \left( K(b) \right)
     - \frac{1}{\rho_*}  \div(n\otimes n) 
    -\div\left( Q(b)n\otimes n \right) 
    + \nabla ( b\widetilde P(b) ).
\end{aligned}\end{equation}
This expression enables us to close necessary product estimates in 
the Besov space of regularity index $s = \frac{d}{2}-3$ 
by a sole restriction $d\ge3$. 

Given initial data $(a_0,m_0)$ such that 
$(\Lambda a_0, m_0) \in \dot{B}^{\frac{d}{2}-1}_{2,1}(\R^d),$ 
we show that the map $\Phi : (b,n) \mapsto (a,m)$ 
defined in \eqref{CNSK.map} with $(a,  m)$ the solution to 
\eqref{CNSKp.scheme} with $\gamma=0,$ i.e., 
\begin{equation}\nonumber
\left\{\begin{aligned}
&\partial_t  a + \div  m = 0 ,\\
&\partial_t  m - \frac{1}{\rho_*} \mathcal{L}  m  
    -\kappa \rho_* \nabla\Delta  a = N(b,n), \\
&(a,  m)_{t=0} =(a_0,m_0),
\end{aligned} \right.
\end{equation}
where $N=N(b,n)$ is as defined in \eqref{CNSKp:mom.nonl.trm.div}, 
has a fixed point in a ball of a suitable metric space. 
To define such a space, for $d\ge3$, we let $\widetilde{CL}_T$ be given by
\begin{equation}\nonumber
\widetilde{CL}_T 
:= \widetilde{L^\infty}(0,T;\dot B^{\frac{d}{2}-3}_{2,1}(\R^d) 
			  \cap \dot B^{\frac{d}{2}-1}_{2,1}(\R^d)) 
		\cap 
    \widetilde{L^1}(0,T;\dot B^{\frac{d}{2}-1}_{2,1}(\R^d)
    			\cap \dot B^{\frac{d}{2}+1}_{2,1}(\R^d))
\end{equation}
equipped with a norm 
\begin{equation}\nonumber\begin{aligned}
\|(a,m)\|_{\widetilde{CL}_T} 
:= \|(\Lambda a, m)\|_{\widetilde{L^\infty_T}(
\dot B^{\frac{d}{2}-3}_{2,1}\cap \dot B^{\frac{d}{2}-1}_{2,1})} 
 + \|\Lambda^2 (\Lambda a, m)
   \|_{\widetilde{L^1_T}(\dot B^{\frac{d}{2}-3}_{2,1}\cap \dot B^{\frac{d}{2}-1}_{2,1})}.
\end{aligned}\end{equation}
By interpolation, we may readily see that $(a,m) \in \widetilde{CL}_T$ 
satisfies \eqref{CLT.intp} and, additionally, 
\begin{equation}\label{tld.CLT.intp}
(\Lambda a, m) \in 
\widetilde{L^q}(0,T;\dot B^{\frac{d}{2}-3+\frac{2}{q}}_{2,1}(\R^d))
\quad\text{for}\quad q \in [1,\infty]. 
\end{equation}
Then we have the following global nonlinear estimates for the case 
$\gamma=0$. 
\begin{lem}\label{l:CNSK.nonl.est.tld}
Let $d\ge 3$ and $T>0$. 
Let $N$ be as defined in \eqref{CNSKp:mom.nonl.trm} and 
let $P$ be a smooth function such that $\gamma=0$. 
Then there exists a positive constant $C$ depending only on 
$\rho_*,$ $P$, $\lambda$, $\mu$, $\kappa$ and $d$  such that 
\begin{equation}\label{l:CNSK.nonl.est.tld:inq1}
\|N(b,n)\|_{\widetilde{L^1_t}(\dot B^{\frac{d}{2}-3}_{2,1}\cap\dot{B}^{\frac{d}{2}-1}_{2,1})} 
\le C (\|(b,n)\|_{\widetilde{CL}_t}^2+\|(b,n)\|_{\widetilde{CL}_t}^3)
\end{equation}
for all $(b,n) \in \widetilde{CL}_T$ and $0 \le t \le T$. 
%
\end{lem}
\begin{rem}
The restriction $d\ge3$ for \eqref{l:CNSK.nonl.est.tld:inq1} is 
used to close product estimates in $\dot B^{\frac{d}{2}-3}_{2,1}(\R^d)$. 
\end{rem}
\begin{proof}
Let $d\ge 3$. 
Note that by \eqref{CLT.intp.pr}, we still have
\begin{equation}\nonumber
(\Lambda b, n) \in 
 \widetilde{L^\infty}(0,T;\dot B^{\frac{d}{2}-1}_{2,1}(\R^d))
\cap \widetilde{L^2}(0,T;\dot B^{\frac{d}{2}}_{2,1}(\R^d))
\cap \widetilde{L^1}(0,T;\dot B^{\frac{d}{2}+1}_{2,1}(\R^d)).
\end{equation}
From this, it is clear that we have 
\eqref{l:CNSK.nonl.est.pr1}, \eqref{l:CNSK.nonl.est.pr2},
\eqref{l:CNSK.nonl.est.pr3} and \eqref{l:CNSK.nonl.est.pr4} as before. 
Furthermore, $(b,n) \in \widetilde{CL}_T$ implies \eqref{tld.CLT.intp}:
\begin{equation}\label{CLT.intp.tld.pr}
(\Lambda b, n) \in 
 \widetilde{L^\infty}(0,T;\dot B^{\frac{d}{2}-3}_{2,1}(\R^d))
\cap \widetilde{L^2}(0,T;\dot B^{\frac{d}{2}-2}_{2,1}(\R^d))
\cap \widetilde{L^1}(0,T;\dot B^{\frac{d}{2}-1}_{2,1}(\R^d)).
\end{equation}
To handle the last term 
$\|\nabla (P(b+\rho_*))\|_{\widetilde{L^1_t}(\dot{B}^{\frac{d}{2}-1}_{2,1})}$, 
we use the above additional low-frequency assumption on $b$. We have 
\begin{equation}\nonumber
\begin{aligned}
\|P'(b+\rho_*) \nabla b \|_{\widetilde{L^1_t}(\dot{B}^{\frac{d}{2}-1}_{2,1})} 
&\le C\| P'(b+\rho_*) \|_{\widetilde{L^1_t}(\dot{B}^{\frac{d}{2}}_{2,1})} 
 \|\nabla b \|_{\widetilde{L^\infty_t}(\dot{B}^{\frac{d}{2}-1}_{2,1})}  \\
&\le C\| b\|_{\widetilde{L^\infty_t}(\dot{B}^{\frac{d}{2}}_{2,1})}
\| b\|_{\widetilde{L^1_t}(\dot{B}^{\frac{d}{2}}_{2,1})},
\end{aligned}\end{equation}
thanks to \eqref{CLT.intp.pr} and \eqref{CLT.intp.tld.pr}. 
Notice that, here, we do not need to use the expression 
\eqref{CNSKp:mom.nonl.trm.div}. 
Combining all information, we have 
\begin{equation}\nonumber\begin{aligned}
\|N(b,n)\|_{\widetilde{L^1_t}(\dot{B}^{\frac{d}{2}-1}_{2,1})}
&\lesssim 
(\|b\|_{\widetilde{L^\infty_t}(\dot{B}^{\frac{d}{2}}_{2,1})} +1)
	\|\nabla b \|_{\widetilde{L^2_t}(\dot{B}^{\frac{d}{2}}_{2,1})} 
	\|n\|_{\widetilde{L^2_t}(\dot{B}^{\frac{d}{2}}_{2,1})}\\
& + \|b \|_{\widetilde{L^\infty_t}(\dot{B}^{\frac{d}{2}}_{2,1})}
	\left(\|\div n \|_{\widetilde{L^1_t}(\dot{B}^{\frac{d}{2}}_{2,1})} 
	+ \|b\|_{\widetilde{L^1_t}(\dot{B}^{\frac{d}{2}}_{2,1})}
	+ \|\Delta b \|_{\widetilde{L^1_t}(\dot{B}^{\frac{d}{2}}_{2,1})}\right) \\
&  +\|n\|_{\widetilde{L^2_t}(\dot{B}^{\frac{d}{2}}_{2,1})}^2 
\left(1 +\|b\|_{\widetilde{L^\infty_t}(\dot{B}^{\frac{d}{2}}_{2,1})} \right).
\end{aligned}\end{equation}

In order to complete the proof of \eqref{l:CNSK.nonl.est.tld:inq1}, 
we still have to bound 
$\|N(b,n)\|_{\widetilde{L^1_t}(\dot B^{\frac{d}{2}-3}_{2,1} )}$
under the assumption that $d\ge3$. 
Thanks to \eqref{CLT.intp.pr}, \eqref{CLT.intp.tld.pr}, Lemma \ref{l:comp}, Lemma \ref{lem:prd.p=2}, we have 
\begin{equation}\label{l:CNSK.nonl.est.tld.pr1}\begin{aligned}
&\|\mathcal{L}\left( Q(b) n\right) \|_{\widetilde{L^1_t}(\dot{B}^{\frac{d}{2}-3}_{2,1})} 
\le C \|\nabla( Q(b)) n + Q(b) \div n \id\|_{\widetilde{L^1_t}(\dot{B}^{\frac{d}{2}-2}_{2,1})} \\
&\le C \|(Q'(b) -Q'(0) )(\nabla b) n \|_{\widetilde{L^1_t}(\dot{B}^{\frac{d}{2}-2}_{2,1})}
	+ \| Q'(0) (\nabla b) n\|_{\widetilde{L^1_t}(\dot{B}^{\frac{d}{2}-2}_{2,1})} 
	+ \|Q(b) \div n \|_{\widetilde{L^1_t}(\dot{B}^{\frac{d}{2}-2}_{2,1})} \\
&\le C \bigg( \|b\|_{\widetilde{L^\infty_t}(\dot{B}^{\frac{d}{2}}_{2,1})} 
	\|\nabla b \|_{\widetilde{L^2_t}(\dot{B}^{\frac{d}{2}}_{2,1})} 
	\|n\|_{\widetilde{L^2_t}(\dot{B}^{\frac{d}{2}-2}_{2,1})}
+  \|\nabla b\|_{\widetilde{L^2_t}(\dot{B}^{\frac{d}{2}}_{2,1})} 
	\|n\|_{\widetilde{L^2_t}(\dot{B}^{\frac{d}{2}-2}_{2,1})}\\
&\qquad\qquad  + \|b \|_{\widetilde{L^\infty_t}(\dot{B}^{\frac{d}{2}}_{2,1})}
		\|n \|_{\widetilde{L^1_t}(\dot{B}^{\frac{d}{2}-1}_{2,1})} \bigg);
\end{aligned}\end{equation}
\begin{equation}\label{l:CNSK.nonl.est.tld.pr2}\begin{aligned}
\|\div \left( K(b) \right) \|_{\widetilde{L^1_t}(\dot{B}^{\frac{d}{2}-3}_{2,1})} 
&\le C \| (\Delta b^2 - |\nabla b|^2) \id - \kappa \nabla a \otimes \nabla a \|_{\widetilde{L^1_t}(\dot{B}^{\frac{d}{2}-2}_{2,1})} \\
&\lesssim 
\| b^2\|_{\widetilde{L^1_t}(\dot{B}^{\frac{d}{2}}_{2,1})} 
+ \| |\nabla b|^2 \id - \kappa \nabla b \otimes \nabla b
         \|_{\widetilde{L^1_t}(\dot{B}^{\frac{d}{2}-2}_{2,1})} \\
&\lesssim 
\| b\|_{\widetilde{L^2_t}(\dot{B}^{\frac{d}{2}}_{2,1})}^2
+ \| \nabla b \|_{\widetilde{L^2_t}(\dot{B}^{\frac{d}{2}-2}_{2,1})}
\| \nabla b \|_{\widetilde{L^2_t}(\dot{B}^{\frac{d}{2}}_{2,1})} \\
&\lesssim 
\| \Lambda b\|_{\widetilde{L^2_t}(\dot{B}^{\frac{d}{2}-1}_{2,1})}^2
+ \| \Lambda b \|_{\widetilde{L^2_t}(\dot{B}^{\frac{d}{2}-2}_{2,1})}
\| \Lambda b \|_{\widetilde{L^2_t}(\dot{B}^{\frac{d}{2}}_{2,1})};
\end{aligned}\end{equation}
\begin{equation}\label{l:CNSK.nonl.est.tld.pr3}
\left\|\frac{1}{\rho_*}  \div(n\otimes n) \right\|_{\widetilde{L^1_t}(\dot{B}^{\frac{d}{2}-3}_{2,1})} 
\le C \| n\otimes n \|_{\widetilde{L^1_t}(\dot{B}^{\frac{d}{2}-2}_{2,1})} 
\le C \|n\|_{\widetilde{L^2_t}(\dot{B}^{\frac{d}{2}}_{2,1})}
         \|n\|_{\widetilde{L^2_t}(\dot{B}^{\frac{d}{2}-2}_{2,1})};
\end{equation}
\begin{equation}\label{l:CNSK.nonl.est.tld.pr4}\begin{aligned}
\|\div\left( Q(b)n\otimes n \right)  \|_{\widetilde{L^1_t}(\dot{B}^{\frac{d}{2}-3}_{2,1})} 
&\le C\|Q(b)n\otimes n \|_{\widetilde{L^1_t}(\dot{B}^{\frac{d}{2}-2}_{2,1})}  \\
&\le C\|Q(b)\|_{\widetilde{L^\infty_t}(\dot{B}^{\frac{d}{2}}_{2,1})}
	\|n\|_{\widetilde{L^2_t}(\dot{B}^{\frac{d}{2}-2}_{2,1})}
	\|n \|_{\widetilde{L^2_t}(\dot{B}^{\frac{d}{2}}_{2,1})} \\
&\le C\|b\|_{\widetilde{L^\infty_t}(\dot{B}^{\frac{d}{2}}_{2,1})}
	\|n\|_{\widetilde{L^2_t}(\dot{B}^{\frac{d}{2}-2}_{2,1})}
	\|n \|_{\widetilde{L^2_t}(\dot{B}^{\frac{d}{2}}_{2,1})};
\end{aligned}\end{equation}
\begin{equation}\label{l:CNSK.nonl.est.tld.pr5}\begin{aligned}
\|  \nabla ( b\widetilde P(b) ) \|_{\widetilde{L^1_t}(\dot{B}^{\frac{d}{2}-3}_{2,1})} 
&\le \|  b\widetilde P(b) \|_{\widetilde{L^1_t}(\dot{B}^{\frac{d}{2}-2}_{2,1})} 
\le C\| b \|_{\widetilde{L^1_t}(\dot{B}^{\frac{d}{2}-2}_{2,1})} 
 \|\widetilde P(b) \|_{\widetilde{L^\infty_t}(\dot{B}^{\frac{d}{2}}_{2,1})}  \\
&\le C \| \Lambda b \|_{\widetilde{L^\infty_t}(\dot{B}^{\frac{d}{2}-3}_{2,1})}
\| \Lambda b\|_{\widetilde{L^1_t}(\dot{B}^{\frac{d}{2}-1}_{2,1})}.
\end{aligned}\end{equation} 
Note here that if we use the expression 
\eqref{CNSKp:mom.nonl.trm}, then instead of 
\eqref{l:CNSK.nonl.est.tld.pr2} and \eqref{l:CNSK.nonl.est.tld.pr5}, we 
end up with using Lemma \ref{lem:prd.p=2} with $s= \frac{d}{2}-3,$ which 
requires $d\ge4.$ 

Summing up \eqref{l:CNSK.nonl.est.tld.pr1}, 
\eqref{l:CNSK.nonl.est.tld.pr2}, 
\eqref{l:CNSK.nonl.est.tld.pr3}, 
\eqref{l:CNSK.nonl.est.tld.pr4} and 
\eqref{l:CNSK.nonl.est.tld.pr5}, 
we have 
\begin{equation}\nonumber\begin{aligned}
\|N(b,n)\|_{\widetilde{L^1_t}(\dot{B}^{\frac{d}{2}-3}_{2,1})}
&\lesssim 
 \|b\|_{\widetilde{L^\infty_t}(\dot{B}^{\frac{d}{2}}_{2,1})} 
 ( \|\nabla b \|_{\widetilde{L^2_t}(\dot{B}^{\frac{d}{2}}_{2,1})} 
	\|n\|_{\widetilde{L^2_t}(\dot{B}^{\frac{d}{2}-2}_{2,1})}
          +\|n\|_{\widetilde{L^2_t}(\dot{B}^{\frac{d}{2}-2}_{2,1})}
	\|n \|_{\widetilde{L^2_t}(\dot{B}^{\frac{d}{2}}_{2,1})} ) \\
& + \|b\|_{\widetilde{L^\infty_t}(\dot{B}^{\frac{d}{2}}_{2,1})} 
	 (\|n \|_{\widetilde{L^1_t}(\dot{B}^{\frac{d}{2}-1}_{2,1})} 
              +\|b \|_{\widetilde{L^1_t}(\dot{B}^{\frac{d}{2}}_{2,1})} ) \\
& + \|n\|_{\widetilde{L^2_t}(\dot{B}^{\frac{d}{2}-2}_{2,1})}
        (\|\nabla b\|_{\widetilde{L^2_t}(\dot{B}^{\frac{d}{2}}_{2,1})} 
           +\|n\|_{\widetilde{L^2_t}(\dot{B}^{\frac{d}{2}}_{2,1})})
  + \|\nabla b \|_{\widetilde{L^\infty_t}(\dot{B}^{\frac{d}{2}-3}_{2,1})}
\| b\|_{\widetilde{L^1_t}(\dot{B}^{\frac{d}{2}}_{2,1})},
\end{aligned}\end{equation}
which completes the proof of \eqref{l:CNSK.nonl.est.tld:inq1}. 
\end{proof}

\begin{proof}[Proof of Theorem \ref{t:L2.WP.szero}]
Similarly to Theorem \ref{t:L2.WP}, 
the proof of Theorem \ref{t:L2.WP.szero} readily follows from 
Corollary \ref{c:lL2}, Lemma \ref{l:CNSK.nonl.est.tld} 
and Lemma \ref{l:fxd.pt.bi.tri.lin}. 
The details are left to the readers. 
\end{proof}

\section{Proof of decay estimates}
In this section, we prove the decay estimates for the 
global small solutions that have been obtained in the previous section. 
\begin{proof}[Proof of Theorem \ref{t:CNSK.dcy}]
The proof below is inspired by \cite{DanXu2017}. 
For notational simplicity, in the following, we denote  
\begin{equation}\nonumber\begin{aligned}
X(t) := \|(a,m)\|_{CL_T} 
=\|((\gamma+\Lambda)  a, m)
		\|_{\widetilde{L^\infty_t}(\dot B^{\frac{d}{2}-1}_{2,1})} 
+ \|\Lambda^2 ((\gamma+\Lambda)  a, m)
	\|_{\widetilde{L^1_t}(\dot B^{\frac{d}{2}-1}_{2,1})}.
\end{aligned}\end{equation}
and
$X_0 := \|(\gamma+\Lambda) a_0, m_0) \|_{\dot{B}^{\frac{d}{2}-1}_{2,1}}.$

Firstly, note that for a solution to \eqref{CNSK} constructed by 
Theorem \ref{t:L2.WP} with $D(t)<\infty,$ we also have
\begin{equation}\label{bnd:Dt.h.inh}
\begin{aligned}
\| \langle\tau\rangle^{\alpha}((\gamma+\Lambda)a, m)
	\|_{\widetilde{L^\infty_t}(\dot B^{\frac{d}{2}-1}_{p,1})}^h 
&\lesssim \|((\gamma+\Lambda)a, m)
	\|_{\widetilde{L^\infty_t}(\dot B^{\frac{d}{2}-1}_{p,1})}^h 
+ \|\tau^{\alpha}((\gamma+\Lambda)a, m)
	\|_{\widetilde{L^\infty_t}(\dot B^{\frac{d}{2}+1}_{p,1})}^h \\
&\lesssim X(t) + D(t) < \infty. 
\end{aligned}\end{equation}
\smallbreak
\noindent\underline{\bf Step 1: Bounds for the low frequencies.} 
We denote the semigroup associated to the left hand-side 
of System \eqref{lCNSKp} by $e^{tK}$ and let 
$$
U = \left[ \begin{array}{c} 
          (\gamma+\Lambda)a \\  m \\
  \end{array} \right], 
\quad
U_0 = \left[ \begin{array}{c} 
          (\gamma+\Lambda) a_0 \\  m_0 \\
  \end{array} \right]
  \quad\text{and}\quad  
F=F (a,m)= \left[ \begin{array}{c} 
         0 \\  N(a,m) \\
  \end{array} \right],
$$
where $N=N(a,m)$ is as defined in \eqref{CNSKp:mom.nonl.trm.div}. 
By Duhamel's formula, we have the expression 
\begin{equation}\nonumber
U (t)= e^{t K} U_0 
+ \int_0^{t} e^{(t-\tau) K} F(a,m)(\tau)d\tau.
\end{equation}

Since $((\gamma+\Lambda) a_0, m_0) \in \dot{B}^{-\frac{d}{2}}_{2,\infty}(\R^d)$, 
the action of the semigroup on $\dot{B}^{s}_{2,1}(\R^d)$ 
given by Lemma \ref{l:lFH.decay} implies 
\begin{equation}\nonumber
\|e^{t K} U_{0} \|_{\dot B^{s}_{2,1}}^l
\le C t^{-\frac{1}{2}(s+\frac{d}{2})} \|((\gamma+\Lambda)a_0, m_0)\|_{\dot B^{-\frac{d}{2}}_{2,\infty}}.
\end{equation}
In addition, we have 
\begin{equation}\nonumber
\|e^{t K} U_{0} \|_{\dot B^{s}_{2,1}}^l
\le \sum_{j\le j_0} 2^{js} \|\dot\Delta_j e^{t K} U_0\|_{L^2}
\le \sum_{j\le j_0} 2^{j (s+\frac{d}{2}) } 2^{-j\frac{d}{2}} 
\|\dot\Delta_j U_0\|_{L^2}
\le C \|U_0 \|_{\dot B^{-\frac{d}{2}}_{2,\infty}}^l
\end{equation}
provided that $s+\frac{d}{2}>0,$ where $C$ 
depends only on $s,$ $d$ and $j_0.$ 
Thus, we have 
\begin{equation}\label{l:lin.est.low.id}
\|e^{t K} U_{0} \|_{\dot B^{s}_{2,1}}^l
\le C \langle t\rangle^{-\frac{1}{2}(s+\frac{d}{2})} 
\|U_0\|_{\dot B^{-\frac{d}{2}}_{2,\infty}}^l.
\end{equation}
By \eqref{l:lin.est.low.id} and the integral representation, we have 
\begin{equation}\label{t:CNSK.dcy.pr.low.est}\begin{aligned}
\|U(t)\|_{\dot B^{s}_{2,1}}^l &= 
\|((\gamma+\Lambda)a, m)(t)\|_{\dot B^{s}_{2,1}}^l
\le \|e^{t K} U_0 \|_{\dot B^{s}_{2,1}}^l
+  \int_0^{t} \| e^{(t-\tau)K} N (a,m)(\tau)\|_{\dot B^{s}_{2,1}}^l d\tau \\
&\le C
\left( \langle t\rangle^{-\frac{1}{2}(s+\frac{d}{2})} 
	\|U_0\|_{\dot B^{-\frac{d}{2}}_{2,\infty}}^l
+  \int_0^{t} \langle t-\tau\rangle^{-\frac{1}{2}(s+\frac{d}{2})} 
	\|N (a,m)(\tau)\|_{\dot B^{-\frac{d}{2}}_{2,\infty}}^l d\tau \right),
\end{aligned}\end{equation}
where the constant $C$ depends on $j_0$ 
(but stays bounded as long as $-\infty<j_0$).  
It now comes down to bounding the nonlinear terms with our prescribed regularity by $X(t)$ and $D(t)$ defined in \eqref{d:CNSK.Dt}. 
We suppose now that $\alpha$ in $D(t)$ is some real number 
to be decided later. 

In what follows, we are going to be frequently using the fact that 
\begin{equation}\label{t:CNSK.dcy.pr.fml1}
\sup_{0 \le \tau \le t} \langle \tau\rangle^{\frac{d}{4}} 
\|((\gamma+\Lambda) a,m)(\tau)\|_{\dot B^0_{2,1}} \lesssim X(t) + D(t). 
\end{equation}
Indeed, we have, thanks to \eqref{bnd:Dt.h.inh},
\begin{equation}\nonumber\begin{aligned}
\|u(\tau)\|_{\dot B^0_{2,1}} 
&\le \|u_l(\tau)\|_{\dot B^0_{2,1}} + \|u_h(\tau)\|_{\dot B^0_{2,1}} 
   \le  \|u_l(\tau)\|_{\dot B^0_{2,1}} + \|u_h(\tau)\|_{\dot B^{\frac{d}{2}-1}_{2,1}} \\
&\lesssim \langle \tau\rangle^{-\frac{d}{4}} 
	\langle \tau\rangle^{\frac{d}{4}} \|u_l(\tau)\|_{\dot B^0_{2,1}} 
+ \langle \tau\rangle^{-\alpha} \langle \tau\rangle^{\alpha} 
	 \|u_h(\tau)\|_{\dot B^{\frac{d}{2}-1}_{2,1}} \\
&\lesssim \langle \tau\rangle^{-\frac{d}{4}}D(t) 
+ \langle \tau\rangle^{-\alpha} (X(t) + D(t)),
\end{aligned}\end{equation}
for $u\in\{(\gamma+\Lambda) a,m\}$ if $\frac{d}{2}-1 \ge0,$ i.e., $d\ge2.$ 
Inequality \eqref{t:CNSK.dcy.pr.fml1} follows if 
\begin{equation}\label{t:CNSK.dcy.pr.ac1}
\alpha\ge\frac{d}{4}.
\end{equation}

\noindent\underline{\sl Estimate of $\mathcal{L} (Q(a) m)$}: 
We have 
\begin{equation}\nonumber\begin{aligned}
\int_0^{t} \langle t-\tau\rangle^{-\frac{1}{2}(s+\frac{d}{2})} 
	\|\mathcal{L} (Q(a) m)(\tau)\|_{\dot B^{-\frac{d}{2}}_{2,\infty}}^l d\tau
&\lesssim \int_0^{t} \langle t-\tau\rangle^{-\frac{1}{2}(s+\frac{d}{2})} 
	\|Q(a(\tau)) m(\tau)\|_{\dot B^{2-\frac{d}{2}}_{2,\infty}} d\tau \\
&\lesssim \int_0^{t} \langle t-\tau\rangle^{-\frac{1}{2}(s+\frac{d}{2})} 
	 \|a(\tau)\|_{\dot B^0_{2,1}} \|m(\tau)\|_{\dot B^2_{2,1}} d\tau
\end{aligned}\end{equation}
thanks to Lemma \ref{l:prd.g.lmt.p2}. 
Splitting the frequencies of $m,$ we have, for the 
integral pertaining $\|m_l(\tau)\|_{\dot B^2_{2,1}},$
\begin{equation}\nonumber\begin{aligned}
\int_0^{t} \langle t-\tau\rangle^{-\frac{1}{2}(s+\frac{d}{2})} 
	 \|a(\tau)\|_{\dot B^0_{2,1}} &\|m_l(\tau)\|_{\dot B^2_{2,1}} d\tau
\lesssim \int_0^{t} \langle t-\tau\rangle^{-\frac{1}{2}(s+\frac{d}{2})} 
	\langle \tau\rangle^{-(1+\frac{d}{2})}  d\tau \\
& \ \times \sup_{0\le\tau \le t} \left(\langle \tau\rangle^{\frac{d}{4}} 
	\|a(\tau)\|_{\dot B^0_{2,1}} \right) 
\sup_{0\le\tau \le t} \left(\langle \tau\rangle^{\frac{1}{2}(2+\frac{d}{2})} 
	\|m(\tau)\|_{\dot B^2_{2,1}}^l \right) \\
&\lesssim \langle t\rangle^{-\frac{1}{2}(s+\frac{d}{2})}  (X(t) + D(t)) D(t), 
\end{aligned}\end{equation}
thanks to \eqref{t:CNSK.dcy.pr.fml1}, Lemma \ref{l:t-est} and 
$\frac{1}{2}(s+\frac{d}{2})<1+\frac{d}{2},$ i.e., 
$s <2 + \frac{d}{2},$ which is automatically 
satisfied since $s \le \frac{d}{2}+1.$
 
As for the high frequencies, we first let $0\le t \le 2$. 
Then we have $\langle t-\tau\rangle \simeq 
\langle \tau\rangle \simeq\langle t\rangle \simeq 1$ 
for all $0\le \tau \le t\le2,$ which gives
\begin{equation}\nonumber
\begin{aligned}
\int_0^{t} \langle t-\tau\rangle^{-\frac{1}{2}(s+\frac{d}{2})} 
	 \|a(\tau)\|_{\dot B^0_{2,1}} \|m_h(\tau)\|_{\dot B^2_{2,1}} d\tau
&\lesssim \int_0^{t} \langle \tau\rangle^{-\frac{d}{4}} \|m_h(\tau)\|_{\dot B^2_{2,1}} d\tau \times D(t) \\
&\lesssim \int_0^{t} \|m(\tau)\|_{\dot B^{\frac{d}{2}+1}_{2,1}} d\tau \times D(t) \\
&\lesssim \langle t\rangle^{-\frac{1}{2}(s+\frac{d}{2})}  X(t) D(t).
\end{aligned}\end{equation}
We next treat the case $t\ge2.$ We split the integrals into $I_1$ and $I_2$ 
as follows:
\begin{equation}\nonumber\begin{aligned}
\int_0^{t} \langle t-\tau&\rangle^{-\frac{1}{2}(s+\frac{d}{2})} 
	 \|a(\tau)\|_{\dot B^0_{2,1}} \|m_h(\tau)\|_{\dot B^2_{2,1}} d\tau \\
&= \left(\int_0^{1}+ \int_1^{t} \right) \langle t-\tau\rangle^{-\frac{1}{2}(s+\frac{d}{2})} 
	 \|a(\tau)\|_{\dot B^0_{2,1}} \|m_h(\tau)\|_{\dot B^2_{2,1}} d\tau =:I_1 + I_2.
\end{aligned}\end{equation}
Note that if $0\le\tau \le1$ and $t\ge2,$ we have 
and $\langle t\rangle \le 2 t \le 2 \langle t-\tau\rangle.$ 
As such, we have 
\begin{equation}\label{t:CNSK.dcy.pr.low2}\begin{aligned}
I_1 &\lesssim 
\int_0^{1} \langle t-\tau\rangle^{-\frac{1}{2}(s+\frac{d}{2})} 
	\langle \tau\rangle^{-\frac{d}{4}} 
	 \|m(\tau)\|_{\dot B^{\frac{d}{2}+1}_{2,1}}^h d\tau \times D(t)\\
&\lesssim  t^{-\frac{1}{2}(s+\frac{d}{2})} 
\int_0^{1}  \|m(\tau)\|_{\dot B^{\frac{d}{2}+1}_{2,1}} d\tau \times D(t)\\
&\lesssim \langle t\rangle^{-\frac{1}{2}(s+\frac{d}{2})} X(t) D(t).
\end{aligned}\end{equation} 
On the other hand, since $\langle \tau\rangle^{\alpha} \le 2^{\alpha} \tau^\alpha$ if $1\le\tau\le t,$ we have 
\begin{equation}\label{t:CNSK.dcy.pr.I2}\begin{aligned}
I_2 \le \int_1^{t} \langle t-\tau&\rangle^{-\frac{1}{2}(s+\frac{d}{2})} 
	 \|a(\tau)\|_{\dot B^0_{2,1}} \|m(\tau)\|_{\dot B^2_{2,1}}^h d\tau \\
&\lesssim
\int_1^{t} \langle t-\tau \rangle^{-\frac{1}{2}(s+\frac{d}{2})} 
	 \langle \tau\rangle^{-\frac{d}{4}-\alpha} 
\left(\langle \tau\rangle^{\alpha} \|m(\tau)\|_{\dot B^{\frac{d}{2}+1}_{2,1}}^h \right)d\tau \times D(t)\\
&\lesssim
\int_1^{t} \langle t-\tau \rangle^{-\frac{1}{2}(s+\frac{d}{2})} 
	 \langle \tau\rangle^{-\frac{d}{4}-\alpha} d\tau \times D(t)^2
\lesssim \langle t \rangle^{-\frac{1}{2}(s+\frac{d}{2})}  D(t)^2
\end{aligned}\end{equation}
by Lemma \ref{l:t-est} and $\frac{1}{2}(s+\frac{d}{2}) \le \frac{d}{4}+\alpha,$
i.e., 
\begin{equation}\label{t:CNSK.dcy.pr.ac2}
s\le 2\alpha.
\end{equation}
Thus, we obtain 
\begin{equation}\label{t:CNSK.dcy.pr.low.est1}
\int_0^{t} \langle t-\tau\rangle^{-\frac{1}{2}(s+\frac{d}{2})} 
	\|\mathcal{L} (Q(a) m)(\tau)\|_{\dot B^{-\frac{d}{2}}_{2,\infty}}^l d\tau
\lesssim  \langle t\rangle^{-\frac{1}{2}(s+\frac{d}{2})}  (X(t) + D(t)) D(t), 
\end{equation}
provided that $\alpha$ satisfies 
\eqref{t:CNSK.dcy.pr.ac1} and \eqref{t:CNSK.dcy.pr.ac2}.  

\smallbreak
\noindent\underline{\sl Estimate of $\div\left( K(a) \right)$}: 
Since $\div \left( K(a) \right) = a \nabla\Delta a,$ we have 
\begin{equation}\nonumber\begin{aligned}
\int_0^{t} \langle t-\tau\rangle^{-\frac{1}{2}(s+\frac{d}{2})} 
	\|\div K(a(\tau)) &\|_{\dot B^{-\frac{d}{2}}_{2,\infty}}^l d\tau
\lesssim \int_0^{t} \langle t-\tau\rangle^{-\frac{1}{2}(s+\frac{d}{2})} 
	\|a \nabla\Delta a \|_{\dot B^{-\frac{d}{2}}_{2,\infty}} d\tau \\
&\lesssim \int_0^{t} \langle t-\tau\rangle^{-\frac{1}{2}(s+\frac{d}{2})} 
	 \|a(\tau)\|_{\dot B^{1}_{2,1}} \| \nabla\Delta a(\tau)\|_{\dot B^{-1}_{2,1}} d\tau
\end{aligned}\end{equation}
thanks to the limiting case of the product rule (Lemma \ref{l:prd.g.lmt.p2}). 
We split the frequencies of the higher order term as before:
$\|a(\tau)\|_{\dot B^2_{2,1}}\le \|a_l(\tau)\|_{\dot B^2_{2,1}}+\|a_h(\tau)\|_{\dot B^2_{2,1}}.$
Note that we have 
$\frac{1}{2}(s+\frac{d}{2})\le \frac{d}{4}+\frac{1}{2}(2+\frac{d}{2}),$
i.e., $s\le 2 + \frac{d}{2},$ which is satisfied by assumption, so 
Lemma \ref{l:t-est} implies 
\begin{equation}\nonumber\begin{aligned}
 \int_0^{t} \langle t-\tau&\rangle^{-\frac{1}{2}(s+\frac{d}{2})} 
	 \|a(\tau)\|_{\dot B^1_{2,1}} \|a_l(\tau)\|_{\dot B^2_{2,1}} d\tau \\
&\lesssim 
 \int_0^{t} \langle t-\tau\rangle^{-\frac{1}{2}(s+\frac{d}{2})} 
\langle \tau\rangle^{-\frac{d}{4}-\frac{1}{2}(2+\frac{d}{2})}  d\tau \times 
(X(t)+D(t))D(t) \\
&\lesssim \langle t\rangle^{-\frac{1}{2}(s+\frac{d}{2})}  (X(t)+D(t))D(t),
\end{aligned}\end{equation}
thanks to \eqref{t:CNSK.dcy.pr.fml1} and the definition of $D(t).$

The bounds pertaining to the term $\|a_h(\tau)\|_{\dot B^2_{2,1}}$ can be 
carried out similarly to the bounds for $\mathcal{L} (Q(a) m).$
Let $0\le t \le 2$ first. Since $\langle t-\tau\rangle \simeq 
\langle \tau\rangle \simeq\langle t\rangle \simeq 1$ 
for all $0\le \tau \le t\le2,$ we have 
\begin{equation}\nonumber\begin{aligned}
\int_0^{t} \langle t-\tau\rangle^{-\frac{1}{2}(s+\frac{d}{2})} 
	 \|a(\tau)\|_{\dot B^1_{2,1}} \|a_h(\tau)\|_{\dot B^2_{2,1}} d\tau
&\lesssim \int_0^{t} \langle \tau\rangle^{-\frac{d}{4}} 
	\|a(\tau)\|_{\dot B^2_{2,1}}^h  d\tau \times D(t) \\
&\lesssim \int_0^{t} \|a(\tau)\|_{\dot B^{\frac{d}{2}+2}_{2,1}}^h d\tau \times D(t) \\
&\lesssim \langle t\rangle^{-\frac{1}{2}(s+\frac{d}{2})}  X(t) D(t)
\end{aligned}\end{equation}
by the definition of $X(t)$ and $D(t).$ 
We next treat the case $t\ge2.$ Splitting the time-integral, we have 
\begin{equation}\nonumber\begin{aligned}
\int_0^{t} \langle t&-\tau\rangle^{-\frac{1}{2}(s+\frac{d}{2})} 
	 \|a(\tau)\|_{\dot B^1_{2,1}} \|a_h(\tau)\|_{\dot B^2_{2,1}} d\tau\\
&= \left(\int_0^{1}+ \int_1^{t} \right) 
	\langle t-\tau\rangle^{-\frac{1}{2}(s+\frac{d}{2})} 
	 \|a(\tau)\|_{\dot B^1_{2,1}} \|a_h(\tau)\|_{\dot B^2_{2,1}}  d\tau 
=:I_3 + I_4.
\end{aligned}\end{equation}
We have 
\begin{equation}\nonumber
I_3 \lesssim  t^{-\frac{1}{2}(s+\frac{d}{2})} 
\int_0^{1}  \|a(\tau)\|_{\dot B^{\frac{d}{2}+2}_{2,1}} d\tau \times D(t)
\lesssim \langle t\rangle^{-\frac{1}{2}(s+\frac{d}{2})} X(t) D(t) 
\end{equation}
by a similar calculation to \eqref{t:CNSK.dcy.pr.low2}. 
On the other hand, for $I_4,$ we have 
\begin{equation}\nonumber
I_4 
\lesssim \langle t \rangle^{-\frac{1}{2}(s+\frac{d}{2})} D(t)^2
\end{equation}
similarly to \eqref{t:CNSK.dcy.pr.I2}, thanks to Lemma \ref{l:t-est} and 
$\frac{1}{2}(s+\frac{d}{2}) \le \frac{d}{4}+\alpha,$
i.e., \eqref{t:CNSK.dcy.pr.ac2}. Thus, we obtain 
\begin{equation}\label{t:CNSK.dcy.pr.low.est2}
\int_0^{t} \langle t-\tau\rangle^{-\frac{1}{2}(s+\frac{d}{2})} 
	\|\div K(a(\tau)) \|_{\dot B^{-\frac{d}{2}}_{2,\infty}}^l d\tau
\lesssim  \langle t\rangle^{-\frac{1}{2}(s+\frac{d}{2})}  (X(t) + D(t)) D(t), 
\end{equation}
provided that $\alpha$ satisfies 
\eqref{t:CNSK.dcy.pr.ac1} and \eqref{t:CNSK.dcy.pr.ac2}.  

\smallbreak
\noindent\underline{\sl Estimate of $ \frac{1}{\rho_*}  \div(m\otimes m)$}: 
We have
\begin{equation}\nonumber\begin{aligned}
\int_0^{t} \langle t-\tau\rangle^{-\frac{1}{2}(s+\frac{d}{2})} 
	\|\div(m\otimes m)(\tau)\|_{\dot B^{-\frac{d}{2}}_{2,\infty}}^l &d\tau 
\lesssim \int_0^{t} \langle t-\tau\rangle^{-\frac{1}{2}(s+\frac{d}{2})} 
	\|(m\otimes m)(\tau) \|_{\dot B^{1-\frac{d}{2}}_{2,\infty}} d\tau \\
&\lesssim \int_0^{t} \langle t-\tau\rangle^{-\frac{1}{2}(s+\frac{d}{2})} 
	 \|m(\tau)\|_{\dot B^0_{2,1}} \|m(\tau)\|_{\dot B^1_{2,1}} d\tau
\end{aligned}\end{equation}
by Lemma \ref{l:prd.g.lmt.p2}. 
Splitting the frequencies of the higher order term by 
$\|m(\tau)\|_{\dot B^1_{2,1}}
\le \|m_l(\tau)\|_{\dot B^1_{2,1}}+\|m_h(\tau)\|_{\dot B^1_{2,1}},$ 
we may carry out the estimates similarly to previous computations. 
Lemma \ref{l:t-est} gives 
\begin{equation}\nonumber\begin{aligned}
 \int_0^{t} \langle t-\tau&\rangle^{-\frac{1}{2}(s+\frac{d}{2})} 
	 \|m(\tau)\|_{\dot B^0_{2,1}} \|m_l(\tau)\|_{\dot B^{1}_{2,1}} d\tau \\
&\lesssim
 \int_0^{t} \langle t-\tau\rangle^{-\frac{1}{2}(s+\frac{d}{2})} 
	\langle \tau\rangle^{-\frac{d}{4}} 
	 \|m_l(\tau)\|_{\dot B^{1}_{2,1}} d\tau \times D(t)\\
&\lesssim 
 \int_0^{t} \langle t-\tau\rangle^{-\frac{1}{2}(s+\frac{d}{2})} 
\langle \tau\rangle^{-\frac{d}{4}-\frac{1}{2}(1+\frac{d}{2})}  d\tau \times 
(X(t)+D(t))D(t) \\
&\lesssim \langle t\rangle^{-\frac{1}{2}(s+\frac{d}{2})} (X(t)+D(t))D(t),
\end{aligned}\end{equation}
thanks to \eqref{t:CNSK.dcy.pr.fml1} and the definition of $D(t).$

For the high frequencies $\|m_h(\tau)\|_{\dot B^1_{2,1}},$ 
let $0\le t \le 2,$ so that we have 
\begin{equation}\nonumber\begin{aligned}
\int_0^{t} \langle t-\tau\rangle^{-\frac{1}{2}(s+\frac{d}{2})} 
	 \|m(\tau)\|_{\dot B^0_{2,1}} \|m_h(\tau)\|_{\dot B^1_{2,1}} d\tau
&\lesssim \int_0^{t} \langle \tau\rangle^{-\frac{d}{4}} 
	\|m(\tau)\|_{\dot B^1_{2,1}}^h  d\tau \ D(t) \\
&\lesssim \int_0^{t} \|m(\tau)\|_{\dot B^{\frac{d}{2}+1}_{2,1}}^h d\tau \ D(t) \\
&\lesssim \langle t\rangle^{-\frac{1}{2}(s+\frac{d}{2})}  X(t) D(t)
\end{aligned}\end{equation}
by the definition of $X(t)$ and $D(t).$ 
We next treat the case $t\ge2.$ Splitting the time-integral, we have 
\begin{equation}\nonumber\begin{aligned}
\int_0^{t} \langle t&-\tau\rangle^{-\frac{1}{2}(s+\frac{d}{2})} 
	 \|m(\tau)\|_{\dot B^0_{2,1}} \|m_h(\tau)\|_{\dot B^1_{2,1}} d\tau\\
&= \left(\int_0^{1}+ \int_1^{t} \right) 
	\langle t-\tau\rangle^{-\frac{1}{2}(s+\frac{d}{2})} 
	 \|m(\tau)\|_{\dot B^0_{2,1}} \|m_h(\tau)\|_{\dot B^1_{2,1}}  d\tau 
=:I_5 + I_6.
\end{aligned}\end{equation}
We have 
\begin{equation}\nonumber
I_5 \lesssim  t^{-\frac{1}{2}(s+\frac{d}{2})} 
\int_0^{1}  \|m(\tau)\|_{\dot B^{\frac{d}{2}+1}_{2,1}} d\tau \times D(t)
\lesssim \langle t\rangle^{-\frac{1}{2}(s+\frac{d}{2})} X(t) D(t) 
\end{equation}
by a similar calculation to \eqref{t:CNSK.dcy.pr.low2}. 
For $I_6,$ we have 
\begin{equation}\nonumber
I_6 \lesssim \langle t \rangle^{-\frac{1}{2}(s+\frac{d}{2})} D(t)^2
\end{equation}
similarly to \eqref{t:CNSK.dcy.pr.I2}, 
thanks to Lemma \ref{l:t-est} and $\frac{1}{2}(s+\frac{d}{2}) \le \frac{d}{4}+\alpha,$
i.e., \eqref{t:CNSK.dcy.pr.ac2}. Thus, we have in the end
\begin{equation}\label{t:CNSK.dcy.pr.low.est3}
\int_0^{t} \langle t-\tau\rangle^{-\frac{1}{2}(s+\frac{d}{2})} 
	\left\| \div(m\otimes m)(\tau) 
		\right\|_{\dot B^{-\frac{d}{2}}_{2,\infty}}^l d\tau
\lesssim  \langle t\rangle^{-\frac{1}{2}(s+\frac{d}{2})}  (X(t) + D(t)) D(t), 
\end{equation}
provided that $\alpha$ satisfies 
\eqref{t:CNSK.dcy.pr.ac1} and \eqref{t:CNSK.dcy.pr.ac2}.  

\smallbreak
\noindent\underline{\sl Estimate of $\div\left( Q(a)m\otimes m \right)$}: 
Lemma \ref{l:prd.g.lmt.p2} and the composition estimate 
(Lemma \ref{l:comp}) give 
\begin{equation}\nonumber\begin{aligned}
\int_0^{t} \langle t-\tau\rangle^{-\frac{1}{2}(s+\frac{d}{2})} &
\|\div\left( Q(a)m\otimes m \right)(\tau) \|_{\dot B^{-\frac{d}{2}}_{2,\infty}}^l d\tau \\
&\lesssim \int_0^{t} \langle t-\tau\rangle^{-\frac{1}{2}(s+\frac{d}{2})} 
	\|(Q(a) m\otimes m)(\tau) \|_{\dot B^{1-\frac{d}{2}}_{2,\infty}} d\tau \\
&\lesssim \int_0^{t} \langle t-\tau\rangle^{-\frac{1}{2}(s+\frac{d}{2})} 
	 \|a(\tau)\|_{\dot B^\frac{d}{2}_{2,1}} \|(m\otimes m)(\tau)\|_{\dot B^{1-\frac{d}{2}}_{2,1}} d\tau.
\end{aligned}\end{equation}
At this point, we notice that $\|a(\tau)\|_{\dot B^\frac{d}{2}_{2,1}} \le X(t)$ and
\begin{equation}\nonumber
 \|m\otimes m\|_{\dot B^{1-\frac{d}{2}}_{2,1}} 
 = \sum_{ij} \| m_i m_j \|_{\dot B^{1-\frac{d}{2}}_{2,1}} 
 \lesssim \sum_{ij} \| \nabla (m_i m_j)\|_{\dot B^{-\frac{d}{2}}_{2,1}}
\end{equation}
by Bernstein inequality, 
so the situation coincides with that of $\div(m\otimes m).$ 
By the same computation as that for \eqref{t:CNSK.dcy.pr.low.est3}, 
we readily obtain
\begin{equation}\label{t:CNSK.dcy.pr.low.est4}
\int_0^{t} \langle t-\tau\rangle^{-\frac{1}{2}(s+\frac{d}{2})} 
	\left\| \div(Q(a) m\otimes m)(\tau) 
		\right\|_{\dot B^{-\frac{d}{2}}_{2,\infty}}^l d\tau
\lesssim  \langle t\rangle^{-\frac{1}{2}(s+\frac{d}{2})}  X(t)(X(t) + D(t)) D(t), 
\end{equation}
provided that $\alpha$ satisfies 
\eqref{t:CNSK.dcy.pr.ac1} and \eqref{t:CNSK.dcy.pr.ac2}.  

\smallbreak
\noindent\underline{\sl Estimate of $(P'(a) - P'(0)) \nabla a$}: 
Lemma \ref{l:prd.g.lmt.p2} and Lemma \ref{l:comp} give 
\begin{equation}\nonumber\begin{aligned}
\int_0^{t} \langle t-\tau\rangle^{-\frac{1}{2}(s+\frac{d}{2})} &
\|(P'(a(\tau)) - P'(0)) \nabla a(\tau)\|_{\dot B^{-\frac{d}{2}}_{2,\infty}}^l d\tau \\
&\lesssim \int_0^{t} \langle t-\tau\rangle^{-\frac{1}{2}(s+\frac{d}{2})} 
	 \|P'(a(\tau)) - P'(0)\|_{\dot B^{0}_{2,1}} \|a(\tau)\|_{\dot B^{1}_{2,1}} d\tau \\
&\lesssim \int_0^{t} \langle t-\tau\rangle^{-\frac{1}{2}(s+\frac{d}{2})} 
	 \|a(\tau)\|_{\dot B^{0}_{2,1}} \|a(\tau)\|_{\dot B^{1}_{2,1}} d\tau.
\end{aligned}\end{equation}
By the same computation as that for \eqref{t:CNSK.dcy.pr.low.est1}, we obtain
\begin{equation}\label{t:CNSK.dcy.pr.low.est5}
\int_0^{t} \langle t-\tau\rangle^{-\frac{1}{2}(s+\frac{d}{2})} 
	\left\| (P'(a(\tau) ) - P'(0)) \nabla a(\tau) 
		\right\|_{\dot B^{-\frac{d}{2}}_{2,\infty}}^l d\tau
\lesssim  \langle t\rangle^{-\frac{1}{2}(s+\frac{d}{2})}  (X(t) + D(t)) D(t), 
\end{equation}
provided that $\alpha$ satisfies 
\eqref{t:CNSK.dcy.pr.ac1} and \eqref{t:CNSK.dcy.pr.ac2}.  

\bigbreak
In the end, inserting \eqref{t:CNSK.dcy.pr.low.est1}, 
\eqref{t:CNSK.dcy.pr.low.est2}, \eqref{t:CNSK.dcy.pr.low.est3}, 
\eqref{t:CNSK.dcy.pr.low.est4} and \eqref{t:CNSK.dcy.pr.low.est5} into 
\eqref{t:CNSK.dcy.pr.low.est}, we deduce that 
\begin{equation}\label{t:CNSK.dcy.pr.low.est.f}\begin{aligned}
\sup_{0\le \tau \le t} \langle \tau\rangle^{\frac{1}{2}(s+\frac{d}{2})} \|U(\tau)\|_{\dot B^{s}_{2,1}}^l \le C
\left( \|U_0\|_{\dot B^{-\frac{d}{2}}_{2,\infty}}^l
+  X(t)(X(t) + D(t)) D(t)\right),
\end{aligned}\end{equation}
given that \eqref{t:CNSK.dcy.pr.ac1} and \eqref{t:CNSK.dcy.pr.ac2} hold, 
where the constant $C$ depends on $d,$ $\mu,$ $\lambda,$ $\kappa$ 
and $\gamma$ and the frequency-threshold $j_0$ can be arbitrary.  

\bigbreak
\noindent\underline{\bf Step 2: Bounds for the high frequencies.} 
Rewriting System \eqref{CNSK} in terms of the weighted unknowns 
$(t^{\alpha} a, t^{\alpha} m),$ we have 
\begin{equation}\label{t:CNSK.dcy.pr.tsys}
\left\{\begin{aligned}
&\partial_t (t^{\alpha} a) + \div (t^{\alpha} m) = \alpha t^{\alpha-1} a, \\
&\partial_t (t^{\alpha} m) - \frac{1}{\rho_{*}} \mathcal{L} (t^{\alpha} m) 
- \kappa\rho_{*} \Delta \nabla(t^{\alpha} a) + \gamma \nabla(t^{\alpha} a) 
= \alpha t^{\alpha-1} m + t^{\alpha} F(a,m), \\
&(t^{\alpha} a, t^{\alpha} m)|_{t=0} = (0,0), 
\end{aligned}\right.
\end{equation}
Using Lemma \ref{l:lFH} with $a \equiv t^{\alpha} a,$  
$m \equiv t^{\alpha} m,$ $f \equiv \alpha t^{\alpha-1} a,$ 
$g \equiv \alpha t^{\alpha-1} v + t^{\alpha} F(a,m)$, 
$s \equiv \frac{d}{2}+1$ and $q\equiv r\equiv \infty,$ we have 
\begin{equation}\nonumber\begin{aligned}
\|\tau^{\alpha} ((\gamma + &\Lambda)a, m)
	\|_{\widetilde{L^\infty_t} (\dot B^{\frac{d}{2}+1}_{2,1})}^h 
&\lesssim
\|(\gamma + \Lambda) (\alpha \tau^{\alpha-1} a) +\alpha \tau^{\alpha-1} m + \tau^{\alpha} F(a,m)
	\|_{\widetilde{L^\infty_t} (\dot B^{\frac{d}{2}-1}_{2,1})}^h,
\end{aligned}\end{equation}
where the constant appearing on the right hand-side 
does not depend on $j_0.$ 

We first handle the lower order linear terms on the right hand-side of the 
above:
\begin{equation}\nonumber
\|(\gamma + \Lambda) (\alpha \tau^{\alpha-1} a) +\alpha \tau^{\alpha-1} m 
	\|_{\widetilde{L^\infty_t} (\dot B^{\frac{d}{2}-1}_{2,1})}^h.
\end{equation}
For $u\in\{(\gamma + \Lambda)a), m\},$ we have 
\begin{equation}\nonumber
\| \alpha \tau^{\alpha-1} u\|_{\widetilde{L^\infty_t} (\dot B^{\frac{d}{2}-1}_{2,1})}^h
\le \| u\|_{\widetilde{L^\infty_t} (\dot B^{\frac{d}{2}-1}_{2,1})} \le X(t)
\end{equation}
when $0\le\tau\le t\le 2$ and
\begin{equation}\nonumber
\| \alpha \tau^{\alpha-1} u
	\|_{\widetilde{L^\infty} (0,1;\dot B^{\frac{d}{2}-1}_{2,1})}^h
\le \| u\|_{\widetilde{L^\infty_t} (\dot B^{\frac{d}{2}-1}_{2,1})} \le X(t)
\end{equation}
when $t\ge2$ and $0\le \tau \le 1.$ 
Moreover, when $t\ge2$ and $1\le\tau\le t,$ 
by the frequency cut-off of the norm, we have
\begin{equation}\nonumber\begin{aligned}
\| \alpha \tau^{\alpha-1} u\|_{\widetilde{L^\infty} (1,t; \dot B^{\frac{d}{2}-1}_{2,1})}^h
&= \alpha \sum_{j\ge j_0} 2^{j(\frac{d}{2}+1)} 2^{-2j} \|\tau^{\alpha-1} u\|_{L^{\infty}(1,t; L^2)}  \\
&\le \alpha 2^{-2j_0} 
\sum_{j\ge j_0} 2^{j(\frac{d}{2}+1)} 
	\|\tau^{\alpha} u\|_{L^{\infty}(0,t; L^2)} \\
&\le \alpha 2^{-2j_0} 
     \| \tau^{\alpha} u\|_{\widetilde{L^\infty_t} (\dot B^{\frac{d}{2}+1}_{2,1})}^h.
\end{aligned}\end{equation}
Taking $j_0\in\Z,$ such that $\alpha 2^{-2j_0} < \frac{1}{4},$ for instance, 
we may absorb the contributions of these terms into the left hand-side. 
As $\alpha$ only depends on the dimension and the parameter $\ep,$ 
so does $j_0.$ Thus, we arrive at 
\begin{equation}\label{t:CNSK.dcy.pr.hgh.est}
\begin{aligned}
\|\tau^{\alpha} ((\gamma + &\Lambda)a, m)
	\|_{\widetilde{L^\infty_t} (\dot B^{\frac{d}{2}+1}_{2,1})}^h 
&\lesssim X(t) + \|\tau^{\alpha} F(a,m)
		\|_{\widetilde{L^\infty_t} (\dot B^{\frac{d}{2}-1}_{2,1})}^h,
\end{aligned}\end{equation}
where the frequency-threshold $j_0$ depends only on the dimension $d$ and the given constants such as $\mu,$ $\lambda,$ $\kappa$ and $\gamma.$ 
It now comes down to estimating the nonlinear terms. 

\smallbreak
\noindent\underline{\sl Estimate of $\mathcal{L} (Q(a) m)$}: 
\begin{equation}\nonumber\begin{aligned}
\|\tau^{\alpha} \mathcal{L} (Q(a) m)
	&\|_{\widetilde{L^\infty_t} (\dot B^{\frac{d}{2}-1}_{2,1})}^h 
\lesssim \sum_{ij} \|\tau^{\alpha} \partial_i \partial_j (Q(a) m)
	\|_{\widetilde{L^\infty_t} (\dot B^{\frac{d}{2}-1}_{2,1})}^h \\
&\lesssim  \sum_{ij}  \|\tau^{\alpha} 
(\partial_i \partial_j  (Q(a)) m + 2 \partial_i (Q(a)) \partial_j m 
+ Q(a) \partial_i\partial_j m ) 
\|_{\widetilde{L^\infty_t} (\dot B^{\frac{d}{2}-1}_{2,1})}^h.
\end{aligned}\end{equation}
We have 
\begin{equation}\nonumber\begin{aligned}
\|\tau^{\alpha} \partial_i \partial_j  (Q(a)) m
	&\|_{\widetilde{L^\infty_t} (\dot B^{\frac{d}{2}-1}_{2,1})} 
\lesssim \|\tau^{\alpha} \partial_i \partial_j  (Q(a)) 
	\|_{\widetilde{L^\infty_t} (\dot B^{\frac{d}{2}}_{2,1})}
\| m\|_{\widetilde{L^\infty_t} (\dot B^{\frac{d}{2}-1}_{2,1})} \\
&\lesssim \|\tau^{\alpha} a \|_{\widetilde{L^\infty_t} (\dot B^{\frac{d}{2}+2}_{2,1})}^h
\| m\|_{\widetilde{L^\infty_t} (\dot B^{\frac{d}{2}-1}_{2,1})} 
\lesssim  X(t) \|\tau^{\alpha} a \|_{\widetilde{L^\infty_t} (\dot B^{\frac{d}{2}+2}_{2,1})};
\end{aligned}\end{equation}
\begin{equation}\nonumber\begin{aligned}
\|\tau^{\alpha}\partial_i (Q(a)) \partial_j m 
	&\|_{\widetilde{L^\infty_t} (\dot B^{\frac{d}{2}-1}_{2,1})} 
\lesssim \| \partial_i (Q(a))
	\|_{\widetilde{L^\infty_t} (\dot B^{\frac{d}{2}-1}_{2,1})}
\| \tau^{\alpha} \partial_j m  \|_{\widetilde{L^\infty_t} (\dot B^{\frac{d}{2}}_{2,1})} \\
&\lesssim \| a\|_{\widetilde{L^\infty_t} (\dot B^{\frac{d}{2}}_{2,1})}
\| \tau^{\alpha}  m\|_{\widetilde{L^\infty_t} (\dot B^{\frac{d}{2}+1}_{2,1})} 
\lesssim X(t) 
 \| \tau^{\alpha}  m\|_{\widetilde{L^\infty_t} (\dot B^{\frac{d}{2}+1}_{2,1})};
\end{aligned}\end{equation}
\begin{equation}\nonumber\begin{aligned}
\|\tau^{\alpha}  Q(a) \partial_i\partial_j m 
	&\|_{\widetilde{L^\infty_t} (\dot B^{\frac{d}{2}-1}_{2,1})} 
\lesssim \|  Q(a) 
	\|_{\widetilde{L^\infty_t} (\dot B^{\frac{d}{2}}_{2,1})}
\| \tau^{\alpha} \partial_i\partial_j m  \|_{\widetilde{L^\infty_t} (\dot B^{\frac{d}{2}-1}_{2,1})} \\
&\lesssim \| a\|_{\widetilde{L^\infty_t} (\dot B^{\frac{d}{2}-1}_{2,1})}
\| \tau^{\alpha}  m\|_{\widetilde{L^\infty_t} (\dot B^{\frac{d}{2}+1}_{2,1})} 
\lesssim X(t)
\|\tau^{\alpha} m\|_{\widetilde{L^\infty_t} (\dot B^{\frac{d}{2}+1}_{2,1})}.
\end{aligned}\end{equation}
Note that we have 
\begin{equation}\nonumber
\|\tau^{\alpha} a \|_{\widetilde{L^\infty_t} (\dot B^{\frac{d}{2}+2}_{2,1})}
\lesssim 
\|\langle \tau\rangle^{\alpha} a \|_{L^\infty_t (\dot B^{\frac{d}{2}+2-\ep}_{2,1})}^l
+ \|\tau^{\alpha} a \|_{\widetilde{L^\infty_t} (\dot B^{\frac{d}{2}+2}_{2,1})}^h
\lesssim D(t),
\end{equation}
provided that $\alpha \le \frac{1}{2} (\frac{d}{2}+2-\ep + \frac{d}{2}),$ i.e., 
\begin{equation}\nonumber
\alpha \le  \frac{1}{2} (d+2-\ep) = \frac{d}{2} +1 -\frac{\ep}{2}.
\end{equation}
We also have 
\begin{equation}\nonumber
\|\tau^{\alpha} m\|_{\widetilde{L^\infty_t} (\dot B^{\frac{d}{2}+1}_{2,1})}
\lesssim
\|\langle\tau\rangle^{\alpha} m\|_{L^\infty_t (\dot B^{\frac{d}{2}+1-\ep}_{2,1})}^l
+ \|\tau^{\alpha} m\|_{\widetilde{L^\infty_t} (\dot B^{\frac{d}{2}+1}_{2,1})}^h
\lesssim D(t),
\end{equation}
provided that $\alpha \le \frac{1}{2} (\frac{d}{2}+1-\ep + \frac{d}{2}),$ i.e., 
\begin{equation}\label{t:CNSK.dcy.pr.ac3}
\alpha \le  \frac{1}{2} (d+1-\ep).
\end{equation}
Thus, we have 
\begin{equation}\label{t:CNSK.dcy.pr.hgh.est1}
\|\tau^{\alpha} \mathcal{L} (Q(a) m)
	\|_{\widetilde{L^\infty_t} (\dot B^{\frac{d}{2}-1}_{2,1})}^h 
\lesssim X(t) D(t),
\end{equation}
provided that $\alpha$ satisfies \eqref{t:CNSK.dcy.pr.ac3}. 

\smallbreak
\noindent\underline{\sl Estimate of $\div\left( K(a) \right)$}: 
Since $\div \left( K(a) \right) = a \nabla\Delta a,$ we have 
\begin{equation}\nonumber\begin{aligned}
\|\tau^{\alpha} \div\left( K(a) \right)
	\|_{\widetilde{L^\infty_t} (\dot B^{\frac{d}{2}-1}_{2,1})}^h 
&= \|\tau^{\alpha} a \nabla\Delta a
	\|_{\widetilde{L^\infty_t} (\dot B^{\frac{d}{2}-1}_{2,1})}^h 
\lesssim
 \| a\|_{\widetilde{L^\infty_t} (\dot B^{\frac{d}{2}}_{2,1})}
 \|\tau^{\alpha} \Delta a
	\|_{\widetilde{L^\infty_t} (\dot B^{\frac{d}{2}}_{2,1})} \\
&\lesssim
X(t)  \|\tau^{\alpha} \Delta a
	\|_{\widetilde{L^\infty_t} (\dot B^{\frac{d}{2}}_{2,1})}. 
\end{aligned}\end{equation}
We have 
\begin{equation}\nonumber\begin{aligned}
  \|\tau^{\alpha} \Delta a \|_{\widetilde{L^\infty_t} (\dot B^{\frac{d}{2}}_{2,1})} 
\lesssim \|\tau^{\alpha} a \|_{L^\infty_t (\dot B^{\frac{d}{2}+2-\ep}_{2,1})}^l 
 + \|\tau^{\alpha} \Delta a \|_{\widetilde{L^\infty_t} (\dot B^{\frac{d}{2}}_{2,1})}^h, 
\end{aligned}\end{equation}
provided that $\alpha \le \frac{1}{2}(\frac{d}{2}+2-\ep + \frac{d}{2}),$ i.e., 
$\alpha \le \frac{1}{2}(d+2-\ep),$ 
which is satisfied if \eqref{t:CNSK.dcy.pr.ac3} holds. 
Thus 
\begin{equation}\label{t:CNSK.dcy.pr.hgh.est2}
\|\tau^{\alpha} \div \left( K(a) \right)
	\|_{\widetilde{L^\infty_t} (\dot B^{\frac{d}{2}-1}_{2,1})}^h 
\lesssim X(t) D(t)
\end{equation}
holds, provided that $\alpha$ satisfies \eqref{t:CNSK.dcy.pr.ac3}.

\smallbreak
\noindent\underline{\sl Estimate of $ \frac{1}{\rho_*}  \div(m\otimes m)$}: 
We have 
\begin{equation}\nonumber\begin{aligned}
\|\tau^{\alpha} \div(m\otimes m)
	\|_{\widetilde{L^\infty_t} (\dot B^{\frac{d}{2}-1}_{2,1})}^h 
&= \|\tau^{\alpha} (m \cdot \nabla m + m \div m)
	\|_{\widetilde{L^\infty_t} (\dot B^{\frac{d}{2}-1}_{2,1})}^h \\
&\lesssim
 \| m\|_{\widetilde{L^\infty_t} (\dot B^{\frac{d}{2}-1}_{2,1})}
 \|\tau^{\alpha} \nabla m
	\|_{\widetilde{L^\infty_t} (\dot B^{\frac{d}{2}}_{2,1})} \\
&\lesssim
X(t)  \|\tau^{\alpha} \nabla m 
	\|_{\widetilde{L^\infty_t} (\dot B^{\frac{d}{2}}_{2,1})}. 
\end{aligned}\end{equation}
As for the decay functional, we deduce 
\begin{equation}\nonumber\begin{aligned}
\|\tau^{\alpha} \nabla m \|_{\widetilde{L^\infty_t} (\dot B^{\frac{d}{2}}_{2,1})} 
\lesssim \|\tau^{\alpha} m \|_{L^\infty_t (\dot B^{\frac{d}{2}+1-\ep}_{2,1})}^l 
 + \|\tau^{\alpha} m \|_{\widetilde{L^\infty_t} (\dot B^{\frac{d}{2}+1}_{2,1})}^h, 
\end{aligned}\end{equation}
$\alpha \le \frac{1}{2}(\frac{d}{2}+1-\ep + \frac{d}{2}),$ i.e.,  
\eqref{t:CNSK.dcy.pr.ac3}. 
Thus, we have 
\begin{equation}\label{t:CNSK.dcy.pr.hgh.est3}
\|\tau^{\alpha} \div(m\otimes m)
	\|_{\widetilde{L^\infty_t} (\dot B^{\frac{d}{2}-1}_{2,1})}^h 
\lesssim X(t) D(t),
\end{equation}
provided that \eqref{t:CNSK.dcy.pr.ac3} holds. 

\smallbreak
\noindent\underline{\sl Estimate of $\div\left( Q(a)m\otimes m \right)$}: 
We have 
\begin{equation}\nonumber\begin{aligned}
\|\tau^{\alpha} \div\left( Q(a)m\otimes m \right)
	\|_{\widetilde{L^\infty_t} (\dot B^{\frac{d}{2}-1}_{2,1})}^h 
&\lesssim\|\tau^{\alpha} Q(a)m\otimes m 
	\|_{\widetilde{L^\infty_t} (\dot B^{\frac{d}{2}}_{2,1})}^h \\
&\lesssim\| a\|_{\widetilde{L^\infty_t} (\dot B^{\frac{d}{2}}_{2,1})} 
\| \tau^{\alpha} m\otimes m 
	\|_{\widetilde{L^\infty_t} (\dot B^{\frac{d}{2}}_{2,1})} \\
&\lesssim X(t)
\sum_{j} \| \tau^{\alpha} \partial_j (m\otimes m) 
	\|_{\widetilde{L^\infty_t} (\dot B^{\frac{d}{2}}_{2,1})}. 
\end{aligned}\end{equation}
By \eqref{t:CNSK.dcy.pr.hgh.est3}, we obtain
\begin{equation}\label{t:CNSK.dcy.pr.hgh.est4}
\|\tau^{\alpha} \div\left( Q(a)m\otimes m \right)
	\|_{\widetilde{L^\infty_t} (\dot B^{\frac{d}{2}-1}_{2,1})}^h 
\lesssim X(t) D(t),
\end{equation}
provided that \eqref{t:CNSK.dcy.pr.ac3} holds. 

\smallbreak
\noindent\underline{\sl Estimate of $(P'(a) - P'(0)) \nabla a$}: 
We have 
\begin{equation}\nonumber\begin{aligned}
\|\tau^{\alpha} (P'(a) - P'(0)) \nabla a
	\|_{\widetilde{L^\infty_t} (\dot B^{\frac{d}{2}-1}_{2,1})}^h 
&\lesssim \|a\|_{\widetilde{L^\infty_t} (\dot B^{\frac{d}{2}-1}_{2,1})} 
\| \tau^{\alpha} \nabla a \|_{\widetilde{L^\infty_t} (\dot B^{\frac{d}{2}}_{2,1})} \\
&\lesssim X(t) \| \tau^{\alpha} \nabla a
	\|_{\widetilde{L^\infty_t} (\dot B^{\frac{d}{2}}_{2,1})},
\end{aligned}\end{equation}
in which we deduce for the decay functional 
\begin{equation}\nonumber\begin{aligned}
 \| \tau^{\alpha} \nabla a \|_{\widetilde{L^\infty_t} (\dot B^{\frac{d}{2}}_{2,1})}
\lesssim \|\tau^{\alpha} a \|_{L^\infty_t (\dot B^{\frac{d}{2}+1-\ep}_{2,1})}^l 
 + \|\tau^{\alpha} a \|_{\widetilde{L^\infty_t} (\dot B^{\frac{d}{2}+1}_{2,1})}^h, 
\end{aligned}\end{equation}
if \eqref{t:CNSK.dcy.pr.ac3} holds. Thus, this implies 
\begin{equation}\label{t:CNSK.dcy.pr.hgh.est5}
\|\tau^{\alpha} \div\left( Q(a)m\otimes m \right)
	\|_{\widetilde{L^\infty_t} (\dot B^{\frac{d}{2}-1}_{2,1})}^h 
\lesssim X(t) D(t),
\end{equation}
provided that \eqref{t:CNSK.dcy.pr.ac3} holds. 

\bigbreak
Inserting 
\eqref{t:CNSK.dcy.pr.hgh.est1}, \eqref{t:CNSK.dcy.pr.hgh.est2}, 
\eqref{t:CNSK.dcy.pr.hgh.est3}, \eqref{t:CNSK.dcy.pr.hgh.est4} and 
\eqref{t:CNSK.dcy.pr.hgh.est5} into \eqref{t:CNSK.dcy.pr.hgh.est}, 
we deduce that  
there exists some frequency-threshold $j_0$ and a constant $C$ 
depending only on $d,$ $\mu,$ $\lambda,$ $\kappa$ and $\gamma$ 
such that the following inequality holds:  
\begin{equation}\label{t:CNSK.dcy.pr.hgh.est.f}
\begin{aligned}
\|\tau^{\alpha} ((\gamma + &\Lambda)a, m)
	\|_{\widetilde{L^\infty_t} (\dot B^{\frac{d}{2}+1}_{2,1})}^h 
& \le C( X(t) + (X(t) + D(t)) D(t) ). 
\end{aligned}\end{equation}

\medbreak
\noindent\underline{\bf Conclusion.}
By \eqref{t:CNSK.dcy.pr.ac1}, \eqref{t:CNSK.dcy.pr.ac2} 
and \eqref{t:CNSK.dcy.pr.ac3}, we now know that $\alpha$ 
must satisfy 
\begin{equation}\nonumber
\max\left\{\frac{d}{4}, \frac{s}{2} \right\}
	\le \alpha \le \frac{1}{2} (d+1-\ep)
\end{equation}
for some arbitrary small $\ep>0.$ 
Thus, the choice of $\alpha =\frac{1}{2} (d+1)-\ep$ is justified 
under our assumption on $s.$  

Putting together \eqref{t:CNSK.dcy.pr.low.est.f} 
and \eqref{t:CNSK.dcy.pr.hgh.est.f}, we finally arrive at 
\begin{equation}\nonumber
D(t) \lesssim D_0 + X(t) +X(t)^2 + D(t)^2. 
\end{equation}
As Theorem \ref{t:L2.WP} ensures that $X(t) \lesssim X_0$ 
with $X_0$ being small, 
we may now conclude that \eqref{t:CNSK.dcy:smllsol} is fulfilled 
for all time provided that $D_0$ and $X_0$ are small enough. 
This concludes Theorem \ref{t:CNSK.dcy}. 
\end{proof}

\bigbreak
\begin{proof}[Proof of Theorem \ref{t:CNSK.dcy.s0}]
As much of the argument is similar to the proof of Theorem \ref{t:CNSK.dcy}, 
we only sketch the proof here. 
As before, for notational simplicity, we denote  
\begin{equation}\nonumber\begin{aligned}
\tilde X(t) := \|(a,m)\|_{\widetilde{CL}_T}  = \|(\Lambda a, m)
		\|_{\widetilde{L^\infty_t}(\dot B^{\frac{d}{2}-1}_{2,1})} 
+ \|\Lambda^2 (\Lambda a, m)
	\|_{\widetilde{L^1_t}(\dot B^{\frac{d}{2}-1}_{2,1})}.
\end{aligned}\end{equation}
and
$\tilde X_0 := \|(\Lambda a_0, m_0) \|_{\dot{B}^{\frac{d}{2}-1}_{2,1}}.$

Note that, for any solution to \eqref{CNSK} with $\gamma=0$ constructed by 
Theorem \ref{t:L2.WP.szero}, if in addition $\tilde D(t)<\infty,$ we have
\begin{equation}\label{bnd:tlDt.h.inh}
\begin{aligned}
\| \langle\tau\rangle^{\alpha}(\Lambda a, m)
	\|_{\widetilde{L^\infty_t}(\dot{B}^{\frac{d}{2}-3}_{2,1}\cap\dot B^{\frac{d}{2}-1}_{p,1})}^h 
&\lesssim \|(\Lambda a, m)
	\|_{\widetilde{L^\infty_t}(\dot{B}^{\frac{d}{2}-3}_{2,1}\cap\dot B^{\frac{d}{2}-1}_{p,1})}^h 
+ \|\tau^{\alpha}(\Lambda a, m)
	\|_{\widetilde{L^\infty_t}(\dot B^{\frac{d}{2}+1}_{p,1})}^h \\
&\lesssim \tilde X(t) + \tilde D(t) < \infty. 
\end{aligned}\end{equation}
In addition, similarly to \eqref{t:CNSK.dcy.pr.fml1}, we have
\begin{equation}\nonumber
\sup_{0 \le \tau \le t} \langle \tau\rangle^{\frac{d}{4}} 
\|(\Lambda a,m)(\tau)\|_{\dot B^0_{2,1}} \lesssim \tilde X(t) + \tilde D(t). 
\end{equation}

Starting from \eqref{t:CNSK.dcy.pr.low.est}, our task is to bound 
\begin{equation}\nonumber
\int_0^{t} \langle t-\tau\rangle^{-\frac{1}{2}(s+\frac{d}{2})} 
	\|N (a,m)(\tau)\|_{\dot B^{-\frac{d}{2}}_{2,\infty}}^l d\tau, 
\end{equation}
provided that \eqref{t:CNSK.dcy.pr.ac1} holds. 

Identical calculations leading to \eqref{t:CNSK.dcy.pr.low.est1},
\eqref{t:CNSK.dcy.pr.low.est2}, \eqref{t:CNSK.dcy.pr.low.est3}, 
\eqref{t:CNSK.dcy.pr.low.est4} and 
\eqref{t:CNSK.dcy.pr.low.est5} yield 
\begin{equation}\label{t:CNSK.dcy.s0.pr.low.est1}
\int_0^{t} \langle t-\tau\rangle^{-\frac{1}{2}(s+\frac{d}{2})} 
	\|\mathcal{L} (Q(a) m)(\tau)\|_{\dot B^{-\frac{d}{2}}_{2,\infty}}^l d\tau
\lesssim  \langle t\rangle^{-\frac{1}{2}(s+\frac{d}{2})} 
	 (\tilde X(t) + \tilde D(t)) \tilde D(t), 
\end{equation}
\begin{equation}\label{t:CNSK.dcy.s0.pr.low.est2}
\int_0^{t} \langle t-\tau\rangle^{-\frac{1}{2}(s+\frac{d}{2})} 
	\|\div K(a(\tau)) \|_{\dot B^{-\frac{d}{2}}_{2,\infty}}^l d\tau
\lesssim  \langle t\rangle^{-\frac{1}{2}(s+\frac{d}{2})}  (\tilde X(t) + \tilde D(t))\tilde D(t), 
\end{equation}
\begin{equation}\label{t:CNSK.dcy.s0.pr.low.est3}
\int_0^{t} \langle t-\tau\rangle^{-\frac{1}{2}(s+\frac{d}{2})} 
	\left\| \div(m\otimes m)(\tau) 
		\right\|_{\dot B^{-\frac{d}{2}}_{2,\infty}}^l d\tau
\lesssim  \langle t\rangle^{-\frac{1}{2}(s+\frac{d}{2})}  (\tilde X(t) +\tilde D(t))\tilde D(t), 
\end{equation}
\begin{equation}\label{t:CNSK.dcy.s0.pr.low.est4}
\int_0^{t} \langle t-\tau\rangle^{-\frac{1}{2}(s+\frac{d}{2})} 
	\left\| \div(Q(a) m\otimes m)(\tau) 
		\right\|_{\dot B^{-\frac{d}{2}}_{2,\infty}}^l d\tau
\lesssim  \langle t\rangle^{-\frac{1}{2}(s+\frac{d}{2})}  \tilde X(t)(\tilde X(t) +\tilde D(t))\tilde D(t) 
\end{equation}
and
\begin{equation}\label{t:CNSK.dcy.s0.pr.low.est5}
\int_0^{t} \langle t-\tau\rangle^{-\frac{1}{2}(s+\frac{d}{2})} 
	\left\| \nabla(a \tilde P(a) )(\tau) 
		\right\|_{\dot B^{-\frac{d}{2}}_{2,\infty}}^l d\tau
\lesssim  \langle t\rangle^{-\frac{1}{2}(s+\frac{d}{2})}  (\tilde X(t) + \tilde D(t))^2, 
\end{equation}
respectively, provided that $\alpha$ satisfies 
\eqref{t:CNSK.dcy.pr.ac1} and \eqref{t:CNSK.dcy.pr.ac2}. 
In the end, by \eqref{t:CNSK.dcy.s0.pr.low.est1}, 
\eqref{t:CNSK.dcy.s0.pr.low.est2}, \eqref{t:CNSK.dcy.s0.pr.low.est3},
\eqref{t:CNSK.dcy.s0.pr.low.est4} and \eqref{t:CNSK.dcy.s0.pr.low.est5}, 
we obtain 
\begin{equation}\label{t:CNSK.dcy.s0.pr.low.est.f}
\sup_{0\le \tau \le t} \langle \tau\rangle^{\frac{1}{2}(s+\frac{d}{2})} \|U(\tau)\|_{\dot B^{s}_{2,1}}^l \le C
\left( \|U_0\|_{\dot B^{-\frac{d}{2}}_{2,\infty}}^l
+ \tilde X(t)(\tilde X(t) +\tilde D(t)) \tilde D(t)\right),
\end{equation}
where the constant $C$ depends on $d,$ $\mu,$ $\lambda$ and $\kappa$ 
and the frequency-threshold $j_0$ can be arbitrary.  

\medbreak
The bounds for the high frequencies are also carried out in the same way 
as the proof for Theorem \ref{t:CNSK.dcy}. 
Applying Lemma \ref{l:lFH} to \eqref{t:CNSK.dcy.pr.tsys} with $\gamma\equiv0$
and employing the same frequency cut-off argument leading to 
\eqref{t:CNSK.dcy.pr.hgh.est}, we obtain 
\begin{equation}\nonumber
\|\tau^{\alpha} (\Lambda a, m)
	\|_{\widetilde{L^\infty_t} (\dot B^{\frac{d}{2}+1}_{2,1})}^h 
\lesssim \tilde X(t) + \|\tau^{\alpha} F(a,m)
		\|_{\widetilde{L^\infty_t} (\dot B^{\frac{d}{2}-1}_{2,1})}^h,
\end{equation}
where the frequency-threshold $j_0$ depends only on the dimension $d$ and the given constants such as $\mu,$ $\lambda,$ $\kappa.$ 

We begin estimating the nonlinear terms in the high-frequency regime. 
This too is carried out in the same way as before. 
By the same computations leading to 
\eqref{t:CNSK.dcy.pr.hgh.est1}, \eqref{t:CNSK.dcy.pr.hgh.est2}, 
\eqref{t:CNSK.dcy.pr.hgh.est3}, \eqref{t:CNSK.dcy.pr.hgh.est4}
we may easily see 
\begin{equation}\label{t:CNSK.dcy.s0.pr.hgh.est1}
\|\tau^{\alpha} \mathcal{L} (Q(a) m)
	\|_{\widetilde{L^\infty_t} (\dot B^{\frac{d}{2}-1}_{2,1})}^h 
\lesssim \tilde X(t) \tilde D(t),
\end{equation}
\begin{equation}\label{t:CNSK.dcy.s0.pr.hgh.est2}
\|\tau^{\alpha} \div \left( K(a) \right)
	\|_{\widetilde{L^\infty_t} (\dot B^{\frac{d}{2}-1}_{2,1})}^h
\lesssim \tilde X(t) \tilde D(t),
\end{equation}
\begin{equation}\label{t:CNSK.dcy.s0.pr.hgh.est3}
\|\tau^{\alpha} \div(m\otimes m)
	\|_{\widetilde{L^\infty_t} (\dot B^{\frac{d}{2}-1}_{2,1})}^h 
\lesssim \tilde X(t) \tilde D(t),
\end{equation}
\begin{equation}\label{t:CNSK.dcy.s0.pr.hgh.est4}
\|\tau^{\alpha} \div\left( Q(a)m\otimes m \right)
	\|_{\widetilde{L^\infty_t} (\dot B^{\frac{d}{2}-1}_{2,1})}^h 
\lesssim \tilde X(t) \tilde D(t),
\end{equation}
respectively, provided that $\alpha$ satisfies \eqref{t:CNSK.dcy.pr.ac3}.

Lastly, to estimate the pressure term, we have to 
slightly modify the calculation of \eqref{t:CNSK.dcy.pr.hgh.est5}. 
We have 
\begin{equation}\label{t:CNSK.dcy.s0.pr.hgh.est5}
\begin{aligned}
\|\tau^{\alpha} (P'(a) - P'(0)) \nabla a
	\|_{\widetilde{L^\infty_t} (\dot B^{\frac{d}{2}-1}_{2,1})}^h 
&\lesssim \|a\|_{\widetilde{L^\infty_t} (\dot B^{\frac{d}{2}}_{2,1})} 
\| \tau^{\alpha} \nabla a \|_{\widetilde{L^\infty_t} (\dot B^{\frac{d}{2}-1}_{2,1})} \\
&\lesssim \tilde X(t) \| \langle\tau\rangle^{\alpha} \nabla a
	\|_{\widetilde{L^\infty_t} (\dot B^{\frac{d}{2}-1}_{2,1})} \\
&\lesssim \tilde X(t) (\tilde X(t) + \tilde D(t)),
\end{aligned}\end{equation}
where we have used \eqref{bnd:tlDt.h.inh} in the last inequality. 
Thus, combining \eqref{t:CNSK.dcy.s0.pr.hgh.est1},
\eqref{t:CNSK.dcy.s0.pr.hgh.est2}, \eqref{t:CNSK.dcy.s0.pr.hgh.est3},
\eqref{t:CNSK.dcy.s0.pr.hgh.est4} and \eqref{t:CNSK.dcy.s0.pr.hgh.est5}, 
we obtain the bound for the high-frequency regime. 
\begin{equation}\label{t:CNSK.dcy.s0.pr.hgh.est.f}
\|\tau^{\alpha} (\Lambda a, m)
	\|_{\widetilde{L^\infty_t} (\dot B^{\frac{d}{2}+1}_{2,1})}^h 
 \le C( \tilde X(t) + (\tilde X(t) + \tilde D(t)) \tilde D(t) ). 
\end{equation}

Now it is clear that \eqref{t:CNSK.dcy.s0:smllsol} follows from 
\eqref{t:CNSK.dcy.s0.pr.low.est.f} and \eqref{t:CNSK.dcy.s0.pr.hgh.est.f}.
This concludes Theorem \ref{t:CNSK.dcy.s0}. 
\end{proof}

\section{Appendix}
\subsection{Calculus facts}
The following inequalities are useful in the proof of decay estimates.
\begin{lem}[\cite{DanXu2017}]\label{l:t-est}
Let $\langle t\rangle:= \sqrt{1+t^2}.$ For any positive real numbers $a,b$, 
there exists a positive constant $C$ such that the 
following inequalities hold:
If $\max(a,b) >1$, then
\begin{equation}\nonumber
\int_{0}^{t} \langle \tau\rangle^{-a} \langle t-\tau\rangle ^{-b}\, d\tau \le C \langle t \rangle^{-\min(a,b)}
\quad\text{for all}\quad t\ge0.
\end{equation}
\end{lem}
Recall that we have $\langle t\rangle \simeq (1+t) \simeq \max(1,t).$ 
In practical situation, we choose either of the above equivalent quantities 
that fit our convenience. 

The following is well-known. 
\begin{lem}[\cite{BCD2011} P73]\label{lem:unifbd}
Let $c_0$ be a positive real number. 
For any positive real number $r$, 
there exists a positive constant $C(r)$ depending on $r$ such that 
\begin{equation}
\nonumber
\sup_{t\ge0} \sum_{k\in\Z} ( 2^{k}t^{\frac{1}{2}} ) ^{r} e^{-c_0 2^{2k}t} \le C(r).
\end{equation}
\end{lem}

\subsection{Praproduct estimates}
We introduce the notation of Bony's paradifferential calculus.
Recall Definition \ref{def:h-LP-proj} and the operators
\begin{equation}\nonumber
\dot \Delta_j u = \mathcal{F}^{-1}[\widehat\phi_j \widehat u], \quad
\widetilde{\dot\Delta}_j u = \sum_{|j-k|\le2} \dot\Delta_k u 
  \quad\text{and}\quad
\dot S_j u =\sum_{j'\le j} \Delta_{j'} u 
\quad\text{for}\quad j\in\Z.
\end{equation}
Given $u, v \in \mathcal{S}'(\R^d)$ 
we have a formal decomposition of the pointwise product $uv$:
\begin{equation}\nonumber
uv = \sum_{j\in\Z} \dot S_{j-3} u \ \dot\Delta_j v
     +\sum_{j\in\Z} \dot S_{j-3} v \ \dot\Delta_j u
     +\sum_{j\in\Z} \widetilde{\dot\Delta}_j u \dot\Delta_{j} v \\
   =: \dot T_u v + \dot T_v u + \dot R(u,v).
\end{equation}
One may easily verify that, with our choice of $\phi,$ the following relations hold:
\begin{equation}\label{supp.property}
  \dot\Delta_j \dot\Delta_k u=0 
     \quad \text{if} \quad |k-j| \ge 3 
\quad \text{and } \quad
  \dot\Delta_k (\dot S_{j-3} u \dot\Delta_j u)=0 
     \quad \text{if} \quad |k-j| \ge 4. 
\end{equation}
Property \eqref{supp.property} readily implies
\begin{equation}\nonumber
\dot\Delta_\ell (\dot T_u v) 
=\sum_{|j-\ell|\le3}\dot\Delta_\ell( \dot S_{j-3} u \ \dot\Delta_j v)
\quad \text{and} \quad
\dot\Delta_\ell (\dot R(u,v)) 
=\sum_{j\ge\ell-4} \dot\Delta_\ell
  (\dot\Delta_j u \widetilde{\dot\Delta}_{j} v).
\end{equation}
For the details of the paradifferential calculus, we refer to 
Chapter 2 of \cite{BCD2011}. 
The up-to-date homogeneous paraproduct estimates are the following. 
\begin{lem}[\cite{AbPa2007}]
\label{l:offdiag}
Let $s_1,s_2 \in \R,$ $1\le p, p_1, p_2 \le\infty,$ $1\le \sigma, \sigma_1,\sigma_2\le\infty$
with $\frac{1}{\sigma}=\frac{1}{\sigma_1}+\frac{1}{\sigma_2}$ and further assume that  
$\frac{1}{p} \le \frac{1}{p_2} + \frac{1}{\lambda} \le 1,$ $1\le \lambda \le \infty$ with $p_1 \le \lambda.$ 
Then there exists a positive constant $C$ 
depending on $d,$ $p,$ $p_1,$ $p_2,$ $s_1,$ $s_2,$ $\sigma_1$ and $\sigma_2,$ such that we have 
\begin{equation}\nonumber
\|\dot T_u v\|_{\dot B^{s_1+s_2+\frac{d}{p}-\frac{d}{p_1}-\frac{d}{p_2}}_{p,\sigma}}
\le C \left\{
	\begin{aligned}
	&\|u\|_{\dot B^{s_1}_{p_1,\sigma_1}}\|v\|_{\dot B^{s_2}_{p_2,\sigma_2}}
		&& \text{if} \quad s_1 + \frac{d}{\lambda}< \frac{d}{p_1}, \\
	&\|u\|_{\dot B^{s_1}_{p_1,1}}\|v\|_{\dot B^{s_2}_{p_2,\sigma}}
		&& \text{if} \quad s_1 + \frac{d}{\lambda}= \frac{d}{p_1}. \\
	\end{aligned}
	\right.
\end{equation}
\end{lem}
\begin{lem}[\cite{AbPa2007}]\label{l:diag}
Let $s_1,s_2 \in \R,$ $1\le p_1, p_2 \le \infty$, $1\le\sigma,\sigma_1,\sigma_2\le\infty$, $\frac{1}{\sigma}=\frac{1}{\sigma_1}+\frac{1}{\sigma_2}$, $\frac{1}{p}\le\frac{1}{p_1}+\frac{1}{p_2}$ and 
$$s_1+s_2+d\inf\left(0,1-\frac{1}{p_1}-\frac{1}{p_2}\right)>0.$$ 
Then there exists a constant $C>0$ depending on 
$d,p, p_1, p_2,\sigma,\sigma_1$ and $\sigma_2$ such that 
\begin{equation*}
\|\dot R(u,v)\|_{\dot B^{s_1+s_2+\frac{d}{p}-\frac{d}{p_1}-\frac{d}{p_2}}_{p,\sigma}}
\le C\|u\|_{\dot B^{s_1}_{p_1,\sigma_1}}
     \|v\|_{\dot B^{s_2}_{p_2,\sigma_1}}
\end{equation*}
Let furthermore $\frac{1}{p}\le\frac{1}{p_1}+\frac{1}{p_2}\le1$ and $s_1+s_2=0$. 
Then there exists a constant $C>0$ depending on 
$d,$ $p,$ $p_1,$ $p_2$ $s_1$ and $s_2$ such that 
\begin{equation*}
\|\dot R(u,v)\|_{\dot B^{\frac{d}{p}-\frac{d}{p_1}-\frac{d}{p_2}}_{p,\infty}}
\le C\|u\|_{\dot B^{s_1}_{p_1,1}}
     \|v\|_{\dot B^{s_2}_{p_2,\infty}}.
\end{equation*}
\end{lem}
Lemma \ref{l:offdiag} and Lemma \ref{l:diag} immediately lead to 
the following product estimates, which are used throughout the paper. 
\begin{lem}\label{lem:prd.p=2}
Let $d\ge1$, $s_1, s_2 \le \frac{d}{2}$ and
$s_1+s_2> 0.$ 
Then it holds that 
\begin{equation*}
\|uv\|_{\dot B^{s_1+s_2-\frac{d}{2}}_{2,1}}
\le C \|u\|_{\dot B^{s_1}_{2,1}}\|v\|_{\dot B^{s_2}_{2,1}}.
\end{equation*}
\end{lem}
For the limiting case $s_1+s_2=0,$ we have the following. 
\begin{lem}\label{l:prd.g.lmt.p2}
Let $s_1,s_2 \in \R$, $s_1 \le \frac{d}{2}$, $s_2 \le \frac{d}{2}$ and $s_1+s_2\ge0.$ 
Then there exists a constant $C>0$ depending on 
$d,$ $s_1$ and $s_2$ such that 
\begin{equation}\nonumber
\| u v \|_{\dot B^{s_1+s_2-\frac{d}{2}}_{2,\infty}}
\le C\|u\|_{\dot B^{s_1}_{2,1}}    \|v\|_{\dot B^{s_2}_{2,1}}.
\end{equation}
\end{lem}

We recall a standard composition estimate in Besov spaces. 
\begin{lem}[\cite{DanXu2017}]
\label{l:comp}
Let $F:I\to \R$ be a smooth function 
(with $I$ an open interval of $\R$ containing 0) 
vanishing at 0. Then for any $s>0$, $1\le p\le\infty$ 
and interval $J$ compactly supported 
in $I$ there exists a constant $C$ such that 
\begin{equation*}
\|F(a)\|_{\dot B^{s}_{p,1}}
\le C \|a\|_{\dot B^{s}_{p,1}}
\end{equation*}
for any $a\in\dot B^{s}_{p,1}(\R^d)$ with values in $J$. 
\smallbreak
In the case $s>-\min(\frac{d}{p},\frac{d}{p'})$ if in addition to the above hypotheses
we have $a\in\dot B^{\frac dp}_{p,1}(\R^d)$ then 
$F(a)\in \dot B^{\frac dp}_{p,1}(\R^d)\cap\dot B^{s}_{p,1}(\R^d)$ and
$$
\|F(a)\|_{\dot B^{s}_{p,1}}\le C \|a\|_{\dot B^{s}_{p,1}}\bigl(|F'(0)|+C\|a\|_{\dot B^{\frac dp}_{p,1}}\bigr).
$$
\end{lem}

\bigbreak
\centerline{\bf Ackowledgement}\smallbreak
The authors are indebted to Professor H. Abels for his valuable comments. 
The first author is supported by JSPS 
Grant-in-Aid for Young Scientists (B) 17K14216 
and Grant-in-Aid for JSPS Research Fellow 18J00557. 
The second author is supported by JSPS Grant-in-Aid for Scientific Research (C) 18K03368 and (B) 16H03945. 

\end{document}